\documentclass[twoside,centertags]{amsart}

\usepackage[cmtip,color,all]{xy}
\usepackage{amsmath}
\usepackage{amsfonts}
\usepackage{amssymb}
\usepackage{amsthm}
\usepackage{graphicx}
\usepackage{hyperref}
\usepackage{euscript}
\usepackage{mathrsfs}

\newtheorem{theorem}{Theorem}[section]
\newtheorem{corollary}[theorem]{Corollary}
\newtheorem{lemma}[theorem]{Lemma}
\newtheorem{proposition}[theorem]{Proposition}

\theoremstyle{definition}
\newtheorem{definition}[theorem]{Definition}

\theoremstyle{definition}
\newtheorem{remark}[theorem]{Remark}
\newtheorem{example}[theorem]{Example}
\newtheorem{notation}[theorem]{Notation}

\newtheorem{claim}[theorem]{Claim}

\newcommand{\Fun}{\mathrm{Fun}}
\newcommand{\C}{\mathscr{C}}

\newcommand{\D}{\mathscr{D}}
\newcommand{\X}{\mathscr{X}}
\newcommand{\Y}{\mathscr{Y}}
\newcommand{\T}{\mathscr{T}}
\newcommand{\Pre}{\mathcal{P}}

\newcommand{\Sh}{\mathrm{Shv}}

\newcommand{\Sp}{\mathcal{S}}
\newcommand{\id}{\mathrm{id}}
\newcommand{\op}{\mathrm{op}}
\newcommand{\map}{\mathrm{map}}

\newcommand{\Set}{\mathrm{Set}}
\newcommand{\colim}{\mathrm{colim}}

\newcommand{\G}{\mathcal{G}}
\newcommand{\lex}{\mathrm{lex}}

\newcommand{\fl}{\mathrm{fl}}
\renewcommand{\1}{\textbf{1}}
\renewcommand{\P}{\mathbb{P}}
\newcommand{\R}{\mathcal{R}}
\begin{document}

\title{Flat functors in higher topos theory} 
\author{George Raptis and Daniel Sch\"appi}
\maketitle

\begin{abstract}
For a small $n$-category $\C$ and an $n$-topos $\X$, we study necessary and sufficient conditions for a functor $f \colon \C \to \X$ to determine a geometric morphism from $\X$ to the $n$-topos $\Pre(\C)_n$ of presheaves on $\C$ for any $n \geq 1$. These results generalize and unify results of Lurie for $n=\infty$ and classical characterizations of flat functors (Diaconescu's theorem) for $n=1$. Interestingly, for $n=\infty$, our analogue of Diaconescu's theorem requires hypercompleteness. As an application, we show that the $\infty$-topos associated to an $n$-site behaves as an $n$-localic $\infty$-topos with respect to hypercomplete $\infty$-topoi. 
\end{abstract}

\section{Introduction}

A geometric morphism $f_* \colon \X \to \T$ between (ordinary Grothendieck) topoi is a right adjoint whose left adjoint is left exact, that is, the left adjoint $f^* \colon \T \to \X$ preserves finite limits. Diaconescu's theorem gives a characterization of geometric morphisms whose target $\T$ is a presheaf category. This characterization is in terms of internal categories and works over an arbitrary base topos (see \cite[Theorem~4.3]{Di}).
 
Let $\C$ be a small (ordinary) category. Then the presheaf category $\Pre(\C)$ is the free cocompletion of $\C$: for any cocomplete category $\X$, the restriction along the Yoneda embedding $j \colon \C \rightarrow \Pre(\C)$ gives an equivalence between left adjoint functors $\Pre(\C) \rightarrow \X$ and arbitrary functors $\C \rightarrow \X$. If $\X$ is in addition a topos, we can further ask for a characterization of those functors $\C \rightarrow \X$ whose induced left adjoint $\Pre(\C) \rightarrow \X$ is left exact. Such a characterization has been given by Mac Lane--Moerdijk in \cite[VII.8-9]{MM}, using the notion of a \emph{filtering} functor \cite[Definition VII.8.1]{MM}. Specifically, it is shown that a functor $f \colon \C \rightarrow \X$ is filtering if and only if the induced left adjoint preserves finite limits \cite[Theorem~VII.9.1]{MM}, yielding a slightly different form of Diaconescu's theorem.

\smallskip
 
Diaconescu's theorem can be used to identify the geometric morphisms with target a presheaf category explicitly in certain cases. For example, if $G$ is a group, considered as a category with one object, then the category of geometric morphisms $\X \rightarrow \Pre(G)$ corresponds to the category of $G$-torsors in $\X$ (see \cite[Example~B.3.2.4(b)]{Jo} for a derivation in terms of Diaconescu's theorem and \cite[Theorem~VIII.2.7]{MM} for a proof using the notion of filtering functor). Another interesting example is given in \cite[VIII.8]{MM}, where it is shown that the presheaf topos $\Pre(\Delta)$ of simplicial sets classifies linear orders \cite[Theorem~VIII.8.5]{MM}. Moreover, Diaconescu's theorem is used in \cite{Mo}  to study what kind of objects are classified by the classifying space $B\C$ of a small category $\C$. This is done by relating continuous maps $X \rightarrow B\C$ (up to homotopy) to geometric morphisms $\Sh(X) \rightarrow \Pre(\C)$ (up to concordance), and then appealing to Diaconescu's theorem.

\medskip
 
Classical topos theory has been successfully generalized to higher categories building on the influential visionary outlook proposed by Grothendieck \cite{Gr}. This program has been carried out in the context of quasi-categories ($\infty$-categories) by Lurie in his landmark work \cite{HTT}. In \cite{HTT}, the notion of (Grothendieck) $n$-topos is defined (see \cite[Definitions~6.1.0.4 and 6.4.1.1]{HTT}) as an accessible left exact localization of the $n$-category of $(n\text{-}1)$-truncated objects in a presheaf category. It is then shown that there exists an axiomatic characterization of $n$-topoi akin to Giraud's theorem in the case $n=1$ (see \cite[Theorems~6.1.0.6 and 6.4.1.5]{HTT}). 
Following the standard convention, we refer here to $(n,1)$-topoi and more generally $(n,1)$-categories simply as $n$-topoi and $n$-categories, respectively. 
 
 A geometric morphism between $n$-topoi is again defined to be a right adjoint whose left adjoint is left exact, that is, it preserves finite limits. Let $\C$ be a small $n$-category. We write $\Pre(\C)_n$ for the $n$-category of presheaves with values in $(n\text{-}1)$-truncated spaces. As for $n=1$, this category is the free cocompletion of $\C$ (see \cite[Theorem~5.1.5.6]{HTT} for the case $n=\infty$; the case $n<\infty$ can be deduced from this, see Proposition~\ref{prop:free_cocomplete_n_category}). Given a cocomplete $n$-category $\Y$, the universal property of $\Pre(\C)_n$ says that any $f \colon \C \rightarrow \Y$ admits a unique extension $\Pre(\C)_n \rightarrow \Y$ which is a left adjoint. This extension is given by the left Kan extension of $f$ along the Yoneda embedding $j_n \colon \C \rightarrow \Pre(\C)_n$ and we denote it by $L(f)_n \colon \Pre(\C)_n \rightarrow \Y$ (or simply $L(f)$ in the case $n=\infty$). Thus a characterization of geometric morphisms from an $n$-topos $\Y$ to $\Pre(\C)_n$ amounts to a characterization of those functors $f \colon \C \rightarrow \Y$ for which $L(f)_n$ preserves finite limits. 

\smallskip

If $n=\infty$ and $\C$ has finite limits, then \cite[Proposition~6.1.5.2]{HTT} shows that $L(f)$ preserves finite limits if and only if $f \colon \C \rightarrow \Y$ preserves finite limits. In fact, the proof shows more (see \cite[Proposition~20.4.3.1]{SAG}): even if $\C$ fails to have all finite limits, the functor $L(f)$ preserves finite limits if and only if it preserves finite limits of representable presheaves. We note that \cite[Proposition~6.1.5.2]{HTT}  is an important ingredient in the proof of the $(n=\infty)$-version of Giraud's theorem \cite[Theorem~6.1.0.6]{HTT}.
 
\smallskip

Now we come to the statements of our main results. First, we show that the above characterization of left exactness of $L(f)$ remains valid for all $n \geq 1$. Note that this is not yet a ``practically intrinsic'' condition on $f$ since it requires understanding of the extension $L(f)_n$.
In Section \ref{section:covering_flat_functors}, we extend the notion of filtering functor from \cite[Definition~VII.8.1]{MM} to the context of $\infty$-categories. Following the terminology introduced in \cite{nLab}, we call these functors \emph{covering-flat}. Then, assuming that $\Y$ is hypercomplete, which is always satisfied if $n < \infty$, our second main result characterizes the functors $f \colon \C \to \Y$ that yield geometric morphisms as the covering-flat functors. The following statement summarizes our results. 

\begin{theorem}\label{main}
Let $\C$ be a small $n$-category and let $\Y$ be an $n$-topos, where $1 \leq n \leq \infty$. For a functor $f \colon \C \to \Y$, the following are equivalent:
\begin{enumerate}
\item[(1)] $L(f)_n \colon \Pre(\C)_n \to \Y$ is left exact; 
\item[(2)] $L(f)_n\colon \Pre(\C)_n \to \Y$ preserves finite limits of representable presheaves, that is, $L(f)_n$ preserves limits of diagrams of the form 
$$(K \to \C \to \Pre(\C)_n)$$ 
where $K$ is a finite simplicial set.
\end{enumerate}
If $n=\infty$, then (1)--(2) are equivalent to:
\begin{enumerate}
\item[(2)$'$] $L(f)=L(f)_{\infty}$ preserves the terminal object and pullbacks of representables, that is, $L(f)$ preserves limits of representables (as in (2)) indexed by the following shapes:
$$
\xymatrix{
\varnothing,  & & (\textsc{Pullbacks}) & \bullet \ar[d] \\
(\textsc{Terminal objects}) & & \bullet \ar[r] & \bullet. 
}
$$
\end{enumerate}
Assuming that $\Y$ is hypercomplete (e.g. when $n<\infty$), then (1)--(2) are also equivalent to:
\begin{enumerate}
\item[(3)] $f\colon \C \to \Y$ is covering-flat (see Definition \ref{covering_condition}). 
\end{enumerate}
\end{theorem}

This is proved in Theorems~\ref{Thm1}  and \ref{Thm2} (for $n=\infty$) and Theorem \ref{Thm3} (for $n < \infty$). It is worth noting that the equivalence $(1) \Leftrightarrow (2)$ in the case $n < \infty$ relies on the characterization $(3)$ in the case $n=\infty$. In other words, we reduce a question in $n$-topos theory for general $n$ to the case $n=\infty$. This reduction turned out to be more convenient than a direct proof for $n < \infty$ along the lines of Theorem~\ref{Thm1} (see Remark~\ref{directproof_n-topoi} for more details).

\smallskip

 We call a functor $f \colon \C \rightarrow \Y$ \emph{flat} if it satisfies the equivalent conditions $(1)$ and $(2)$ of the above theorem. If the target $\Y$ is not hypercomplete, then the notions of flat and covering-flat functors will differ in general, i.e., there are covering-flat functors which are not flat (Remark~\ref{hypercompleteness}). The equivalence with $(2)'$ is only valid in the case $n=\infty$: already for $n=1$ there are examples where this equivalence does not hold (Example~\ref{example_groupoid}).
 
\medskip
 
 As an application of Theorem~\ref{main}, we show that the $\infty$-category of sheaves $\Sh(\C, \tau)$ on an $n$-category $\C$ with a Grothendieck topology $\tau$ behaves as an $n$-localic $\infty$-topos relative to hypercomplete $\infty$-topoi (Subsection \ref{application}).  Given two $n$-topoi $\X$ and $\Y$, we write $\mathrm{Fun}_{\ast}(\Y,\X)$ for the $\infty$-category of geometric morphisms. For $m \geq n$, an $m$-topos $\X$ is \emph{$n$-localic} if for any $m$-topos $\Y$, the restriction functor
 \[
 \mathrm{Fun}_{\ast}(\Y,\X) \rightarrow \mathrm{Fun}_{\ast}(\Y_n,\X_n)
 \]
 is an equivalence, where we write $\Y_n$ (respectively $\X_n$) for the $n$-topos of $(n\text{-}1)$-truncated objects in $\Y$ (respectively $\X$). It is known that $\Pre(\C)$ is $n$-localic in the case where the $n$-category $\C$ has finite limits \cite[6.4.5]{HTT}, but this fails for a general $n$-category $\C$. However, we show in Corollary~\ref{cor:flat_functors4} that the restriction functor
 \[
 \mathrm{Fun}_{\ast}\bigl(\Y,\Sh(\C,\tau)_m \bigr) \rightarrow \mathrm{Fun}_{\ast} \bigl( \Y_n, \Sh(\C,\tau)_n \bigr)
 \]
 is an equivalence if $\Y$ is a \emph{hypercomplete} $m$-topos. In particular, the $\infty$-topos $\Sh(\C,\tau)$ can be regarded as $n$-localic with respect to hypercomplete $\infty$-topoi.

\bigskip

\noindent \textbf{Acknowledgements.} We thank Denis-Charles Cisinski for his useful comments and Sebastian Wolf for interesting discussions. We also thank Peter Haine for an illuminating discussion concerning $n$-localic $\infty$-topoi. The authors were partially supported by \emph{SFB 1085 -- Higher Invariants} (University of Regensburg), funded by the DFG.

\section{Flat functors} \label{sec:infinity-topoi}

Let $\C$ be a small $\infty$-category. We denote by $\Pre(\C)$ the $\infty$-topos of presheaves on $\C$ and we denote the Yoneda embedding by $j \colon \C \to \Pre(\C)$ (see \cite[5.1]{HTT}, \cite[5.8]{Ci}).

\smallskip

Let $\X$ be a cocomplete $\infty$-category and let $f \colon \C \to \X$ be a functor. Then there is an essentially unique colimit-preserving functor \cite[5.1.5]{HTT}, \cite[6.3]{Ci}
$$L(f) \colon \Pre(\C) \to \X$$
which extends $f$ along the Yoneda embedding $j \colon \C \to \Pre(\C)$; in addition, $L(f)$ admits a right adjoint \cite[5.2.6.5]{HTT}, \cite[6.3.4]{Ci}, \cite[4.1.4]{NRS}. 

In the case where $\X$ is an $\infty$-topos, our goal is to specify necessary and sufficient conditions on $f \colon \C \to \X$ such that the left adjoint $L(f) \colon \Pre(\C) \to \X$ is left exact and therefore determines a geometric morphism between $\infty$-topoi. The following theorem gives a characterization of such left exact left adjoint functors $L(f) \colon \Pre(\C) \to \X$ in terms of finite limits of representable presheaves. This characterization was shown in \cite[6.1.5.2]{HTT} (when $\C$ admits finite limits) and in \cite[20.4.3.1]{SAG} (for general $\C$). This result plays an important role in the proof of the $\infty$-categorical version of Giraud's theorem (see \cite[6.1.0.6]{HTT}).  

\begin{theorem} \label{Thm1}
Let $\C$ be a small $\infty$-category and let $\X$ be an $\infty$-topos. For a functor $f \colon \C \to \X$, the following are equivalent:
\begin{enumerate}
\item $L(f) \colon \Pre(\C) \to \X$ is left exact; 
\item $L(f)$ preserves finite limits of representable presheaves, that is, $L(f)$ preserves limits of diagrams of the form 
$$(K \to \C \to \Pre(\C))$$ 
where $K$ is a finite simplicial set.
\item[(2)$'$] $L(f)$ preserves the terminal object and pullbacks of representable presheaves, that is, $L(f)$ preserves limits as in (2) indexed by the following shapes:
$$
\xymatrix{
\varnothing,  & & (\textsc{Pullbacks}) & \bullet \ar[d] \\
(\textsc{Terminal objects}) & & \bullet \ar[r] & \bullet. 
}
$$
\end{enumerate}
\end{theorem}
\begin{proof} (1) $\Rightarrow$ (2) $\Rightarrow$ (2)$'$ are obvious. (2)$'$ $\Rightarrow$ (1) was shown in \cite[20.4.3.1, 20.4.3.5]{SAG} (following the proof of \cite[6.1.5.2]{HTT}, which remains valid even though $\C$ may not admit finite limits). We give a slightly different proof which circumvents the theory of free groupoids developed in \cite[6.1.4]{HTT}. 

For the implication (2)$'$ $\Rightarrow$ (1), it suffices to prove that the functor $L(f)$ preserves pullbacks in $\Pre(\C)$. Given an object $Z$ in $\Pre(\C)$, we say that the functor $L(f)$ is \emph{left exact at $Z$} if it preserves pullback squares whose lower right corner is $Z$. The first step in the proof of (2)$'$ $\Rightarrow$ (1) is the following reduction:

\begin{claim} \label{Claim1}
$L(f)$ is left exact at $Z \in \Pre(\C)$ if and only if $L(f)$ preserves pullbacks in $\Pre(\C)$ of the form: 
$$
\xymatrix{
W \ar[r] \ar[d] & j(c) \ar[d] \\
j(c') \ar[r] & Z
}
$$
where $c, c' \in \C$.
\end{claim}
\noindent \emph{Proof of Claim \ref{Claim1}.} The proof is contained in the proof of \cite[6.1.5.2]{HTT} and uses (twice) the fact that $\Pre(\C)$ and $\X$ have universal colimits and that $L(f)$ preserves colimits. Indeed for a general pullback diagram in $\Pre(\C)$
$$
\xymatrix{
W \ar[r] \ar[d] & X \ar[d] \\
Y \ar[r] & Z
}
$$
we may write $X$ (and then $Y$) as a small colimit of representable presheaves, and then the aforementioned properties of $\Pre(\C)$, $\X$, and $L(f)$, allow us to reduce the assertion that $L(f)$ preserves this pullback square to the assertion that $L(f)$ preserves the associated pullback squares in which 
$X$ (and then $Y$) are replaced by representable presheaves. \qed

\smallskip

\noindent Thus by Claim \ref{Claim1}, it suffices to prove that $L(f)$ preserves pullbacks in $\Pre(\C)$ of the form:
$$
\xymatrix{
W \ar[r] \ar[d] & j(c) \ar[d] \\
j(c') \ar[r] & Z.
}
$$
By assumption, this holds when $Z$ is representable. Following the proof of \cite[6.1.5.2]{HTT}, we claim that the class of objects $Z \in \Pre(\C)$ for which $L(f)$ is left exact at $Z$ is closed under small coproducts.  This is proved in the same way as in \cite[6.1.5.2]{HTT} and it is a consequence of the fact that coproducts are disjoint (see \cite[6.1.5.1]{HTT}). Thus, $L(f)$ is left exact at small coproducts of representable presheaves.

For an arbitrary presheaf $Z$ in $\Pre(\C)$, let
$$q  \colon \bigsqcup_{\alpha \colon j(c) \rightarrow Z} j(c) \longrightarrow Z$$ be the canonical morphism from the coproduct which is indexed by a set of morphisms $(\alpha \colon j(c) \rightarrow Z)$ in $\Pre(\C)$, one from each homotopy class of morphisms in $\map(j(c), Z) \simeq Z(c)$ (see \cite[5.8]{Ci}, \cite[5.5.2]{HTT}), for all objects $c$ of $\C$. 

Let $U_{\bullet}$ be the \v{C}ech nerve associated to the morphism $q$. Since $U_0$ is a coproduct of representable presheaves, $L(f)$ is left exact at $U_0$. It follows from this that $L(f)(U_{\bullet})$ is again a groupoid object. The morphism $q \colon U_0 \rightarrow Z$ is an effective epimorphism, since it is $\pi_0$-surjective objectwise by construction. Therefore, the augmentation $q$ exhibits $Z$ as a colimit of $U_\bullet$ in the $\infty$-topos $\Pre(\C)$. The functor $L(f)$ preserves colimits, so $L(f)(Z)$ is also the colimit of the groupoid object $L(f)(U_\bullet)$ in $\X$. The fact that groupoids in the $\infty$-topos $\X$ are effective implies that $L(f)(U_{\bullet})$ is the \v{C}ech nerve of the morphism (augmentation) $L(f)(q)$. It follows that $L(f)$ sends the pullback diagram
 $$
 \xymatrix{
 U_1 \ar[r] \ar[d] & U_0 \ar[d]^q \\ U_0 \ar[r]_-{q} & Z }
 $$
 to a pullback square in $\X$.
 
 Now let $c,c^{\prime} \in \C$ be arbitrary objects in $\C$ and let $j(c) \rightarrow Z$ and $j(c^{\prime}) \rightarrow Z$ be two morphisms. By the construction of $U_0$, these two morphisms factor through $q \colon U_0 \rightarrow Z$ (up to a choice of homotopy).  The two pullback diagrams labelled $(\ast)$ in the diagram
 \[
 \xymatrix{ W \ar[dd] \ar@{}[rdd]|{(\ast)} \ar[r] & V \ar@{}[rd]|{(\ast)} \ar[d] \ar[r] &  j(c) \ar[d] \\ & U_1 \ar[r] \ar[d] & U_0 \ar[d]^q \\ j(c^{\prime}) \ar[r] & U_0 \ar[r]_q &Z }
 \]
 are preserved by $L(f)$, since $L(f)$ is left exact at small coproducts of representable presheaves. We just showed that the lower left pullback square is also preserved by $L(f)$, hence it follows that $L(f)$ preserves the outer pullback square as well. 
 
Thus Claim~\ref{Claim1} implies that $L(f)$ preserves arbitrary pullback squares. Since $L(f)$ also preserves the terminal object by assumption, it follows that $L(f)$ preserves finite limits. This completes the proof of (2)$'$ $\Rightarrow$ (1).
\end{proof}

\begin{definition}[Flat functor] \label{def:flat_functor}
Let $\C$ be a small $\infty$-category and let $\X$ be an $\infty$-topos. A functor $f \colon \C \to \X$ is called \emph{flat} if it satisfies the equivalent conditions of Theorem \ref{Thm1}. We denote by $\Fun^{\fl}(\C, \X)$ the $\infty$-category of flat functors. 
\end{definition}

\begin{remark}
 The term ``flat functor'' originated in the context of additive categories (see the introduction of \cite{OR}). The left Kan extension $L(f)$ can classically be described as a coend, which is also known as a functor tensor product. Flatness of a functor is thus a natural generalization of the notion of a flat module.
 
 In the context of (1-)topoi, a characterization of flat functors was first given by Diaconescu \cite[4.3]{Di}. The flat functors are characterized by the fact that the total space of their classifying discrete fibration is a filtered category in the internal sense \cite[\S 2]{Di}. In \cite[VII.8--9]{MM}, the notion of \emph{filtering} functors is introduced and shown to coincide with the notion of flatness.
  
 In \cite{Ka}, Karazeris observed that the internal notion of flatness (that is, flatness expressed in terms of geometric logic) makes sense as long as the target category is a category equipped with a Grothendieck topology. In \cite{nLab}, the resulting notion is called a \emph{covering-flat} functor. This notion subsumes all the other notions of flat functor mentioned so far as special cases, see \cite{nLab} for details. In particular, if the target category is a topos endowed with the canonical topology, then the functor is covering-flat if and only if it is filtering. In Section \ref{section:covering_flat_functors}, we will adapt the notion of covering-flat functors to the context of $\infty$-categories. As we will see, in this context, the class of covering-flat functors may differ from the class of flat functors. We do not pursue the direct analogue of Diaconescu's characterization in terms of internal categories.
\end{remark}

For a small $\infty$-category $\C$ and a cocomplete $\infty$-category $\X$, let $\Fun^L(\Pre(\C), \X)$ denote the $\infty$-category of colimit-preserving functors (or, equivalently, left adjoint functors; see \cite[4.1.4]{NRS}). The universal property of the Yoneda embedding $j \colon \C \to \Pre(\C)$ (see \cite[5.1.5]{HTT}, \cite[6.3]{Ci}) shows that the restriction functor
\begin{equation} \label{Yoneda_emb} \tag{*}
\Fun^L(\Pre(\C), \X) \xrightarrow{j^*} \Fun(\C, \X)
\end{equation}
is an equivalence of $\infty$-categories. Let $\Fun^{L, \lex}(\Pre(\C), \X) \subset \Fun^L(\Pre(\C), \X)$ denote the full subcategory spanned by the left exact left adjoint functors. The characterization of Theorem \ref{Thm1} specializes the equivalence \eqref{Yoneda_emb} to the following equivalence between $\infty$-categories of functors.

\begin{corollary} \label{cor:flat_functors}
Let $\C$ be a small $\infty$-category and let $\X$ be an $\infty$-topos. Then restriction along the Yoneda embedding $j \colon \C \to \Pre(\C)$ defines an equivalence of $\infty$-categories  
$$\Fun^{L, \lex}(\Pre(\C), \X) \xrightarrow{j^*} \Fun^{\fl}(\C, \X).$$
In other words, the $\infty$-topos $\Pre(\C)$ is the classifying $\infty$-topos for flat functors. 
\end{corollary}

Moreover, we also recover the case considered in \cite[6.1.5.2]{HTT} as a special case of Theorem \ref{Thm1}.

\begin{corollary} \label{cor:finite_lim}
Let $\C$ be a small $\infty$-category with finite limits and let $\X$ be an $\infty$-topos. For a functor $f \colon \C \to \X$, the following are equivalent:
\begin{enumerate}
\item $L(f) \colon \Pre(\C) \to \X$ is left exact; 
\item $f$ preserves finite limits. 
\end{enumerate} 
\end{corollary}
\begin{proof}
(1) $\Rightarrow$ (2) follows from the fact that $j \colon \C \to \Pre(\C)$ preserves finite limits. (2) $\Rightarrow$ (1) follows from Theorem \ref{Thm1} using that $\C$ has finite limits and $j$ preserves them. 
\end{proof}

\begin{example} \label{infty_grpd}
Let $\G$ be a small $\infty$-groupoid, $\X$ an $\infty$-topos and let $f \colon \G \to \X$ be a functor. Then $L(f) \colon \Pre(\G) \to \X$ is left exact if and only if the colimit of $f$ is the terminal object in $\X$. This follows from Theorem \ref{Thm1} using the observation that pullbacks along equivalences are preserved by any functor and the fact that the terminal object in $\Pre(\G)$ is the colimit of the Yoneda embedding $j \colon \G \to \Pre(\G)$. (The latter is true for presheaves on any small $\infty$-category $\C$. To see this, first note that the colimit of $j$ is given by the left Kan extension of $j$ along the unique functor $\C \rightarrow \Delta^0$. This Kan extension can be computed in two steps, by first taking the left Kan extension along $j \colon \C \rightarrow \Pre(\C)$ and then taking the left Kan extension along $\Pre(\C) \rightarrow \Delta^0$. The left Kan extension of $j$ along itself is the identity, so the second left Kan extension computes the colimit of $\id_{\Pre(\C)}$. Then the claim follows because the colimit of the identity functor is given by the terminal object.)

In addition, assuming that $\G$ has a single object and $\X = \Sp$ is the $\infty$-topos of spaces, then this condition is equivalent to the condition that $f$ is given by the free $\G$-space $\G$, that is, the corepresentable functor $\G \to \Sp$. Moreover, in this case, the $\infty$-category $\Fun^{L, \lex}(\Pre(\G), \Sp) \simeq \Fun^{\fl}(\G, \Sp)$ is equivalent to $\G$ (by the Yoneda lemma).
\end{example}

\section{Covering-flat functors}\label{section:covering_flat_functors}

\subsection{Diaconescu's theorem for $\infty$-topoi} The difficulty with the characterizations of flat functors $f \colon \C \to \X$ obtained in Theorem \ref{Thm1} is that they involve the functor $L(f) \colon \Pre(\C) \to \X$. 

The purpose of this section is to obtain another characterization of flat functors (Theorem \ref{Thm2}) which depends directly only on $f$ and is analogous to the classical characterization in Diaconescu's theorem \cite{Di, MM} for the context of ordinary topoi. This characterization is based on the following definition of \emph{covering-flat} functors (which is a direct generalization of the notion for 1-categories introduced in \cite[VII.8.1, VII.8.4]{MM}). We will define this notion for general functors $f \colon \C \to \X$ where $\X$ is an $\infty$-category which admits finite limits, but we will only be interested in the case where $\X$ is an ($n$-)topos. 

First we recall some well known terminology. Given an $\infty$-category $\X$ which admits finite limits, we say that a morphism $(q \colon u \to y)$ in $\X$ is an \emph{effective epimorphism} if the right Kan extension of $u$ to an augmented (semi-)simplicial object (\v{C}ech nerve) is a colimit diagram (cf.\ \cite[6.2.3]{HTT}).  We say that a set of morphisms $\{y'_i \to y\}_{i \in I}$ is a \emph{covering family} of $y \in \X$ if the canonical map
$$\bigsqcup_{i \in I} y'_i \longrightarrow y$$
is an effective epimorphism in $\X$.

\begin{definition} \label{covering_condition} Let $\C$ be a small $\infty$-category, let $\X$ be an $\infty$-category which admits finite limits, and let $f \colon \C \to \X$ be a functor.
\begin{itemize}
\item[(a)] For any simplicial set $K$, any $K$-diagram $g \colon K \to \C$ and any cone on the composite diagram $h = f g \colon \C \to \X$, 
$$h^{\triangleleft} \colon K^{\triangleleft} \to \X,$$
with cone object $y \in \X$, let $S_{g, h^{\triangleleft}}$ denote the collection of morphisms in $\X$
$$(y' \xrightarrow{t} y)$$
with the following property: the cone $ h^{\triangleleft} \circ t$ obtained (essentially uniquely) by precomposition of $h^{\triangleleft}$ with $t$ admits a 
morphism of cones 
 in $\X_{/h}$ to a cone of the form $fg^{\triangleleft} \colon K^{\triangleleft} \to \C \xrightarrow{f} \X$ for some cone $g^{\triangleleft} \colon K^{\triangleleft} \to \C$ on $g$. 
\item[(b)] We say that $f$ is \emph{covering-flat} if for any finite simplicial set $K$, any $K$-diagram $g \colon K \to \C$ and any cone $h^{\triangleleft} \colon K^{\triangleleft} \to \X$ on $h = f g$ with cone object $y \in \X$, there is a covering familiy of $y \in \X$ which consists of morphisms $(y' \to y)$ in $S_{g, h^{\triangleleft}}$. 
\end{itemize}
\end{definition}  

\begin{remark}
 Note that $S_{g, h^{\triangleleft}}$ is a sieve on the cone object $y \in \X$. If $\X$ is an $\infty$-topos, we can thus rephrase the condition of being covering-flat as follows: for each $g \colon K \rightarrow \C$ as above, the sieve $S_{g, h^{\triangleleft}}$ is a covering sieve for the canonical topology on $\X$. (In \cite[6.2.4.1]{HTT}, the canonical topology is only defined on a \emph{small} $\infty$-category $\C$ equipped with a left exact functor $f \colon \C \rightarrow \X$, but we can also apply this definition to the case $f=\id_{\X}$. The proof that this construction gives a Grothendieck topology \cite[6.2.4.2]{HTT} does not rely on the smallness of $\C$.) We can naturally extend this notion to functors whose target is a site. If $f \colon \C \rightarrow \D$ is a functor and $\D$ is equipped with a Grothendieck topology, then we call $f$ \emph{covering-flat} if for each for each $g \colon K \rightarrow \C$ as above and each cone $h^{\triangleleft}$ on $h=fg$, the sieve $S_{g, h^{\triangleleft}}$ belongs to the Grothendieck topology on $\D$. This notion generalizes the notion of covering-flatness considered in \cite{Ka, nLab} for 1-categories. 
\end{remark} 
 
In the main cases of interest, the definition of a covering-flat functor can be simplified as follows. In this article, we will only consider the case where $\X$ is a $n$-topos. 

\begin{lemma} \label{lem_cov-flat}
Let $f \colon \C \to \X$ be as in Definition \ref{covering_condition} where $\X$ is a presentable $\infty$-category with universal colimits. Then $f$ is covering-flat if (and only if) $f$ satisfies the condition of Definition \ref{covering_condition}(b) for those triples $(K, g \colon K \to \C, h^{\triangleleft} \colon K^{\triangleleft} \to \X)$ for which $h^{\triangleleft}$ is a limit cone of $h = fg \colon \C \to \X$. 
\end{lemma}

\begin{proof}
Given $(K, g \colon K \to \C, h^{\triangleleft} \colon K^{\triangleleft} \to \X, y)$ as in Definition~\ref{covering_condition}(b), we may factor $h^{\triangleleft}$ through the limit cone $(fg)^{\triangleleft} \colon K^{\triangleleft} \to \X$ of $fg$ with cone object $\overline{y}\in \X$. By assumption, we may find a covering family $\{y'_i \to \overline{y}\}_{i \in I}$ given by morphisms in $S_{g, (fg)^{\triangleleft}}$. Pulling back this covering family along the morphism of cone objects $y \to \overline{y}$, we obtain the required covering family of $y \in \X$ (using that coproducts and effective epimorphisms in $\X$ are preserved under pullback; see the proof of \cite[6.2.3.15]{HTT}) by morphisms which now lie in $S_{g, h^{\triangleleft}}$. 
\end{proof}

Thus a covering-flat functor $f \colon \C \to \X$ to an $\infty$-topos $\X$ is roughly a functor for which finite limits of diagrams $K \xrightarrow{g} \C \xrightarrow{f} \X$ factor locally through images of cones on $g$. This condition can be regarded as a type of weak left exactness locally. 

\begin{proposition} \label{flat_vs_cov-flat}
Let $\C$ be a small $\infty$-category and let $\X$ be an $\infty$-topos. If $f \colon \C \to \X$ is a flat functor (Definition \ref{def:flat_functor}), then $f$ is also covering-flat (Definition \ref{covering_condition}).
\end{proposition}
\begin{proof}
Let $(K, g \colon K \to \C, h^{\triangleleft} \colon K^{\triangleleft} \to \X, y)$ be as in Definition \ref{covering_condition}(b) where $h^{\triangleleft}$ is a limit cone (using Lemma \ref{lem_cov-flat}). If $f \colon \C \to \X$ is flat, then $L(f)$ preserves finite limits by Theorem \ref{Thm1}, so we may assume that $h^{\triangleleft}$ is the cone $L(f) \circ (jg)^{\triangleleft}$ where $(jg)^{\triangleleft} \colon K^{\triangleleft} \to \Pre(\C)$ is a limit cone on $jg \colon K \to \C \to \Pre(\C)$ with cone object denoted by $Z \in \Pre(\C)$. We recall that $L(f)$ preserves effective epimorphisms because $L(f)$ preserves both colimits and finite limits, so it preserves \v{C}ech nerves and their colimits. Then the image under $L(f)$ of a covering family $\{q_i \colon j(c_i) \rightarrow Z\}_{i \in I}$ of $Z$ by representable presheaves defines a covering family of $L(f)(Z)=y$. Moreover, each such morphism $(j(c) \to Z)$ corresponds (using the Yoneda embedding) to a cone $g^{\triangleleft} \colon K^{\triangleleft} \rightarrow \C$ on $g$. This shows that the covering family $\{L(f)(q_i) \colon f(c_i) \to y\}_{i \in I}$ is in $S_{g, h^{\triangleleft}}$ as required.
\end{proof}

Our main result is that the converse of Proposition \ref{flat_vs_cov-flat} also holds. Interestingly, this requires hypercompleteness (see Remark \ref{hypercompleteness}). We refer to \cite[6.5]{HTT} for a detailed account of hypercomplete objects and hypercompletion in higher topos theory.   
 
\begin{theorem} \label{Thm2}
Let $\C$ be a small $\infty$-category and let $\X$ be a hypercomplete $\infty$-topos. Then a functor $f \colon \C \to \X$ is flat if and only if $f$ is covering-flat.
\end{theorem}

\begin{remark} \label{hypercompleteness} The hypercompleteness assumption of Theorem \ref{Thm2} is necessary. This is not surprising given that the notion of covering-flat functor depends only on the effective epimorphisms of $\X$. For an explicit example, consider a small $\infty$-category $\D$ with a Grothendieck topology $\tau$ such that the associated $\infty$-topos of sheaves $\X := \Sh(\D, \tau)$ is not hypercomplete (see \cite[6.5.4]{HTT}). Let $x \in \X$ be an $\infty$-connective object which is not terminal in $\X$. Then the functor $f \colon \ast \to \X$, $f(\ast) = x$, is clearly not flat, but $f$ is covering-flat. To see this, it suffices to prove that for any finite simplicial set $K$, the limit $y := x^K$ of the constant $K$-diagram at the object $x$ admits an effective epimorphism  
$$(y' \to y)$$
with the property that the composite cone with cone object $y'$ factors through the trivial cone with cone object $x$. Note that the canonical morphism $(x \to y)$ clearly has this factorization property. Moreover, $(x \to y)$ is an effective epimorphism because this property can be detected by passing to the $n$-truncations (see \cite[7.2.1.14]{HTT}) and these yield terminal objects in this case.
\end{remark}

\subsection{Canonical semi-simplicial hypercoverings} \label{criteria_flatness} 
The proof of Theorem \ref{Thm2} is based on an elementary criterion for flatness which uses the construction of certain (semi-)simplicial hypercoverings in $\Pre(\C)$.

Let $\X$ be an $\infty$-topos and $X \in \X$. A simplicial diagram $U_{\bullet} \colon \Delta^{\mathrm{op}} \rightarrow \X_{ \slash X}$ is called a hypercovering of $X$ if the morphisms
\[
U_n \rightarrow (\mathrm{cosk}_{n-1} U_\bullet)_n
\]
are effective epimorphisms for all $n \geq 0$ (see \cite[6.5.3.2]{HTT}). A hypercovering of $X$ is called \emph{effective} if its colimit is $X$ (that is, the terminal object of $\X_{\slash X}$).  An $\infty$-topos is hypercomplete if and only if every hypercovering is effective \cite[6.5.3.12]{HTT}.

In practice, it is often simpler to construct semi-simplicial hypercoverings instead. These are indexed by the category $\Delta_{<}^{\mathrm{op}}$, where $\Delta_{<}$ denotes the subcategory of $\Delta$ consisting of injective order-preserving maps. A semi-simplicial object $U_{\bullet}$ of $\X_{\slash X}$ is called a \emph{semi-simplicial hypercovering} if for all $n \geq 0$, the canonical morphism
 \[
 U_n \rightarrow \underset{[k] \hookrightarrow [n], k < n}{\mathrm{lim}} U_k
\]
 is an effective epimorphism (see \cite[A.5.2.3, A.5.1.6]{SAG}). By \cite[A.5.3.3]{SAG}, the colimit of a semi-simplicial hypercovering $U_{\bullet}$ of $X \in \X$ is $\infty$-connective in $\X_{/X}$. Thus, in the case where $\X$ is hypercomplete, we again find that the colimit of a semi-simplicial hypercovering $U_{\bullet}$ in $\X_{\slash X}$ is the terminal object given by $X$. 

\begin{notation}
Let $\C$ be a small $\infty$-category, let $\X$ be an $\infty$-topos, and let $f \colon \C \to \X$ be a functor. Given a diagram of representable presheaves $K \xrightarrow{g} \C \to \Pre(\C)$ with limit denoted by $Z(g) \in \Pre(\C)$, we write $X(g) \in \X$ for the limit of the composite diagram $K \xrightarrow{g} \C \xrightarrow{f} \X$ and
$$c_g \colon L(f)\bigl(Z(g)\bigr) \to X(g)$$
for the canonical comparison morphism in $\X$. Given a semi-simplicial object (or any other diagram) $U_{\bullet}$ in $\Pre(\C)_{/Z(g)}$, we obtain a semi-simplicial object $L(f)_* U_{\bullet}$ in $\X_{/X(g)}$ given by application of the functor $L(f)$ and postcomposition with $c_g$.
\end{notation}

\begin{lemma} \label{criterion_flatness}
Let $\C$ be a small $\infty$-category, let $\X$ be a hypercomplete $\infty$-topos, and let $f \colon \C \to \X$ be a functor. Suppose that for every finite diagram of representable presheaves $$K \xrightarrow{g} \C \to \Pre(\C),$$ there is a (semi-simplicial) hypercovering $U^g_{\bullet}$ in $\Pre(\C)_{/Z(g)}$ with the property that $L(f)_* U^g_{\bullet}$ is again a (semi-simplicial) hypercovering in $\X_{/X(g)}$. Then $f$ is a flat functor (i.e., $L(f)$ is left exact). 
\end{lemma}
\begin{proof}
By Theorem \ref{Thm1}, it suffices to show that $L(f)$ preserves finite limits of representable presheaves. Let $K \xrightarrow{g} \C \xrightarrow{j} \Pre(\C)$ be a diagram of representable presheaves, where $K$ is a finite simplicial set, and let $Z(g) \in \Pre(\C)$ denote a limit of $jg$. We need to show that the canonical morphism
$$c_g \colon L(f)(Z(g)) \longrightarrow X(g) := \mathrm{lim}_K (K \xrightarrow{g} \C \xrightarrow{f} \X)$$
is an equivalence in $\X$. By assumption, there is a (semi-simplicial) hypercovering $U^g_{\bullet}$ in $\Pre(\C)_{/Z(g)}$ with the property that $L(f)_* U^g_{\bullet}$ is a (semi-simplicial) hypercovering in $\X_{/X(g)}$. Since $\Pre(\C)$ is hypercomplete, the associated augmented (semi-) simplicial object in $\Pre(\C)$, also denoted by $U^g_{\bullet}$, is a colimit diagram with colimit $Z(g) \in \Pre(\C)$. The functor $L(f)$ preserves colimits, so $L(f)(U^g_{\bullet})$ is a colimit diagram in $\X$ with colimit $L(f)(Z(g))$. On the other hand, since $\X$ is hypercomplete, the (semi-simplicial) hypercovering $L(f)_* U^g_{\bullet}$ in $\X_{/X(g)}$ also defines a colimit diagram in $\X$ with colimit $X(g)$. It follows that the comparison map $c_g$ must be an equivalence, as required.
\end{proof} 

\begin{remark} \label{variations_criterion_flatness}
We mention some useful variations and simplifications of Lemma \ref{criterion_flatness}. First, by Theorem \ref{Thm1}, it suffices to consider only the cases of the terminal object and pullbacks (instead of an arbitrary finite simplicial set $K$). Second, the criterion applies to any functor $f \colon \C \to \X$, without any hypercompleteness assumptions on $\X$, if we require instead that the induced (semi-)simplicial object $L(f)_* U^g_{\bullet}$ determines a colimit diagram (the proof is the same).   
\end{remark}

In order to apply Lemma~\ref{criterion_flatness}, we need a general construction of a semi-simplicial hypercovering of a presheaf in $\Pre(\C)$. We will construct this hypercovering inductively, using the following standard technique for constructing (homotopy coherent) diagrams on inverse categories.

 Recall that a finite (1-)category $I$ is called an \emph{inverse category} if there exists a function $\mathrm{deg} \colon \mathrm{Ob}(I) \rightarrow \{0, \ldots, n\}$, $n \geq 0$, called the degree function, with the property that any non-identity morphism in $I$ lowers the degree. We are mainly interested in the full subcategory of $\Delta_{<}^{\op}$ spanned by $[k]$ with $k \leq n$;  this is a (finite) inverse category with degree function which sends $[k]$ to $k$. The following lemma shows that we can construct diagrams on such inverse categories iteratively. If we restrict attention to 1-categories, this is a special case of the iterative construction of diagrams indexed by a Reedy category.
 
 \begin{lemma}\label{lemma:inverse_category}
 Let $I$ be a finite inverse (1-)category with degree function 
 \[
 \mathrm{deg} \colon \mathrm{Ob}(I) \rightarrow \{0, \ldots, n\}
 \]
 and let $i \in I$ be an object with $\mathrm{deg}(i)=n$. Let $J$ be the full subcategory on $\mathrm{Ob}(I) \setminus \{i\}$.
 
 Let $\C$ be an $\infty$-category with finite limits and let $U \colon J \rightarrow \C$ be a diagram. Then an extension of $U$ to a diagram $I \rightarrow \C$ is uniquely determined by an object $U_i \in \C$ and a morphism $U_i \rightarrow M_i U$, where $M_i U$ denotes the limit of the diagram in $\C$
 \[
J_{i/} \xrightarrow{(i \to j) \mapsto j} J \xrightarrow{U} \C
 \]
 and $J_{i/} = J \times_I I_{i/}$. The extension of the diagram $U$ sends the object $i \in I$ to $U_i$.
 \end{lemma}
 
 \begin{proof}
 Consider the pullback diagram of simplicial sets (here and elsewhere we identify ordinary categories with their nerves)
 \[
 \xymatrix{J_{i/} \ar[r] \ar[d] & J \ar[d] \\
  I_{i/} \ar[r] & I}
 \]
whose vertical morphisms are induced by the natural inclusions and whose horizontal morphisms are induced by the codomain functors. We claim that this diagram is also a pushout diagram of simplicial sets. Since the vertical morphisms are monomorphisms, this boils down to the following claim: the codomain functor $I_{i/} \rightarrow I$ induces a bijection between the $k$-simplices of $I_{i/}$ that are not in the image of $J_{i/} \rightarrow I_{i/}$ and the $k$-simplices of $I$ that are not in the image of $J \subset I$. Such $k$-simplices are in both cases given by chains of morphisms which at some point pass through $i$. In fact, since non-identity morphisms lower the degree, these chains must start at the object $i$. For the same reason, such a $k$-simplex in $I_{i/}$ has to start at the object $\id_i \colon i \rightarrow i$. The claim readily follows from these observations.
 
Next we note that there exists an isomorphism $I_{i/} \cong \Delta^0 \ast J_{i/}$ with the property that $J_{i/} \rightarrow I_{i/}$ corresponds to the natural inclusion in the join. Indeed, any $k$-simplex not in the image of $J_{i/} \rightarrow I_{i/}$ is either degenerate on the vertex $\id_i$ or it is obtained from a unique pair of a degenerate $\ell$-simplex on $\id_i$ and an $m$-simplex in $J_{i/}$ with $\ell + m+1=k$, by concatenating the two in the evident manner.
 
 Combining these two facts, we find that an extension of $U$ corresponds to a morphism $\Delta^0 \ast J_{i/} \rightarrow \C$ whose restriction to $J_{i/}$ is the composite
 \[
  \xymatrix{ J_{i/} \ar[r] & J \ar[r]^{U} & \C.}
 \]
 In other words, such an extension corresponds to a cone on this diagram in $\C$. Since $\C$ has finite limits, such a cone  is determined up to equivalence by a morphism $U_i \rightarrow M_i U$. Moreover, the object $(\id_i \colon i \rightarrow i)$ in $I_{i/}$ corresponds to the cone object 
of $\Delta^0 \ast J_{i/}$ and is sent to $i$ by the codomain functor, so any such extension of $U$ sends $i$ to the object $U_i \in \C$.
\end{proof}

We can now proceed with our construction of a semi-simplicial hypercovering of a fixed presheaf in $\Pre(\C)$.
Given $Z \in \Pre(\C)$, we construct a semi-simplicial hypercovering $U^Z_{\bullet} \colon N(\Delta^{op}_{<}) \to \Pre(\C)_{/Z}$ inductively as follows. For $n=0$, we set 
$$U^Z_0 := \bigsqcup_{\alpha \colon j(c) \to Z} j(c)$$
where the coproduct is indexed by a set of morphisms $(\alpha \colon j(c) \to Z)$ in $\Pre(\C)$, one from each homotopy class in $\Pre(\C)$. There is an obvious canonical map $U_0 \to Z$ in $\Pre(\C)$, which is an effective epimorphism by construction. Assuming that the semi-simplicial object $U^{Z}_{\bullet}$ in $\Pre(\C)_{/Z}$ has been constructed for $\bullet < n$, we consider the $n$-th \emph{matching object}
$$M_n U^Z_{\bullet} := \underset{[k] \hookrightarrow [n], k < n}{\mathrm{lim}} U^Z_k \in \Pre(\C)_{/Z}$$
(note that the limit is formed in $\Pre(\C)_{/Z}$) and then define
$$U^Z_n := \bigsqcup_{\alpha \colon j(c) \to M_n U^Z_{\bullet}} j(c)$$
where the coproduct is again indexed by a set of morphisms $(\alpha \colon j(c) \to M_n U_{\bullet}^Z)$ in $\Pre(\C)$, one from each homotopy class of such morphisms. There is a canonical morphism $U^Z_n \to M_n U^Z_{\bullet}$, which is an effective epimorphism by construction. By Lemma~\ref{lemma:inverse_category}, 
this morphism determines essentially uniquely an extension of $U^Z_{\bullet}$ to all $\bullet \leq n$. This completes the construction of the 
semi-simplicial hypercovering $U^Z_{\bullet}$ in $\Pre(\C)_{/Z}$ for any presheaf $Z$.

Specializing Lemma \ref{criterion_flatness} to these choices of semi-simplicial hypercoverings yields the following criterion for flatness. 

\begin{corollary} \label{criterion_flatness2}
Let $\C$, $\X$ and $f \colon \C \to \X$ be as in Lemma \ref{criterion_flatness}. Suppose that for every finite diagram of representable presheaves $$K \xrightarrow{g} \C \to \Pre(\C)$$ with limit $Z(g)$ in $\Pre(\C)$, the associated semi-simplicial object $$L(f)_* U^{Z(g)}_{\bullet} \text{ in } \X_{/X(g)}$$ is a semi-simplicial hypercovering. Then $f$ is a flat functor (i.e., $L(f)$ is left exact). 
\end{corollary}

We note that Remark \ref{variations_criterion_flatness} applies also to Corollary \ref{criterion_flatness2}. 

\smallskip

Since the semi-simplicial hypercoverings $U^Z_{\bullet}$ are given degreewise by coproducts of representable presheaves, the criterion of Corollary~\ref{criterion_flatness2} ultimately translates into a condition on the functor $f$. In order to extract a useful description of this condition and compare this with the notion of covering-flat functors, we will need an explicit identification of the matching objects of these semi-simplicial hypercoverings.

\subsection{A decomposition of certain limits in $\Pre(\C)$} It will be convenient to consider the following more general setting. Let $L$ be a finite simplicial set and let $U \colon L \to \Pre(\C)$ be a diagram in $\Pre(\C)$ whose values are coproducts of representable presheaves, i.e., for every $0$-simplex/object $\ell$ in $L$, we have:
\begin{equation} \label{coprod_rep} \tag{$\sqcup$}
U(\ell) \simeq \bigsqcup_{i \in I_{\ell}} j(c_{i}).
\end{equation}
The unit of the adjunction 
$$\mathrm{colim} \colon \Pre(\C) \rightleftarrows \Sp \colon \delta$$ 
yields a morphism between $L$-diagrams in $\Pre(\C)$,
$$U \longrightarrow U^{\delta} := \delta \circ \mathrm{colim} \circ U,$$
and we have an identification
$$U^{\delta}(\ell) \simeq \bigsqcup_{I_{\ell}} \1$$
where $\1$ denotes the terminal object in $\Pre(\C)$, that is, the constant presheaf at the terminal object $\ast \in \Sp$. (This follows from the fact that $\mathrm{colim} \colon \Pre(\C) \rightarrow \Sp$ is the essentially unique functor such that $\colim \circ j$ is constant with value the terminal object of $\Sp$.)

Let $c \colon \Pre(\C) \to \Pre(\C)^L$ denote the constant $L$-diagram functor. A \emph{special point of $U$} is a morphism of $L$-diagrams in $\Pre(\C)$ of the form:
$$\alpha \colon c (\1) \longrightarrow U^{\delta}$$
which is given pointwise by an inclusion of summand, i.e., for any $\ell \in L$, the corresponding component of $\alpha$, 
$$\alpha(\ell) \colon \1 \to U^{\delta}(\ell) \simeq \bigsqcup_{I_{\ell}} \1,$$
agrees with one of the canonical inclusions into the coproduct. Note that $c(\1)$ and $U^{\delta}$ are $0$-truncated objects in $\Pre(\C)^L$ and each special point $\alpha$ of $U$ is a monomorphism. Let $\P(U)$ denote the set of (homotopy classes of) special points of $U$.

\begin{lemma} \label{decomposition_points}
Let $U \colon L \to \Pre(\C)$ be a diagram as above.  For each special point $(\alpha \colon c(\1) \to U^{\delta})$, we define $U_{\alpha} \colon L \to \Pre(\C)$ by the pullback 
$$
\xymatrix{
U_{\alpha} \ar[d] \ar[r] & U \ar[d] \\
c(\1) \ar[r]^{\alpha} & U^{\delta}.
}
$$ 
\begin{itemize}
\item[(a)] The morphism $U_{\alpha} \to U$ is a monomorphism.
\item[(b)] The values of $U_{\alpha}$ are representable presheaves in $\Pre(\C)$ and so the diagram $U_{\alpha}$ factors (essentially uniquely) through the Yoneda embedding $\C \xrightarrow{j} \Pre(\C)$,
$$U_{\alpha} \colon L \xrightarrow{U'_{\alpha}} \C \xrightarrow{j} \Pre(\C).$$
\item[(c)]  For $\alpha \neq \alpha'$ in $\P(U)$, the square in $\Pre(\C)$
$$
\xymatrix{
\mathbf{0} \ar[r] \ar[d] & \mathrm{lim}_L U_{\alpha} \ar[d] \\
\mathrm{lim}_L U_{\alpha'} \ar[r] & \mathrm{lim}_L U
}
$$
is a pullback. (Here $\mathbf{0}$ denotes the initial object in $\Pre(\C)$, that is, the constant presheaf at the initial object $\varnothing \in \Sp$.)
\item[(d)] The canonical morphism 
$$\bigsqcup_{\alpha \in \P(U)} \mathrm{lim}_L U_{\alpha} \longrightarrow \mathrm{lim}_L U$$
is an equivalence in $\Pre(\C)$.
\end{itemize}
\end{lemma} 
\begin{proof}
(a) holds because $\alpha \colon c(\1) \to U^{\delta}$ is a monomorphism and monomorphisms are closed under pullbacks in $\Pre(\C)$. For (b), first recall that for each $\ell \in L$, the morphism in $\Pre(\C)$
$$\alpha_{\ell} \colon \1 \to \coprod_{I_{\ell}} \1$$
corresponds to the inclusion of a summand for some $i \in I_{\ell}$. Then, $U_{\alpha}(\ell) \simeq j(c_i)$, using that coproducts in $\Pre(\C)$ are disjoint. The second claim of (b) follows directly from the fact that $j \colon \C \to \Pre(\C)$ is full and faithful (\cite[5.1]{HTT}, \cite[6.3]{Ci}).  

For (c), consider the diagram $U_{\alpha} \times_U U_{\alpha'} \colon L \to \Pre(\C)$ and the induced pullback in $\Pre(\C)$
$$
\xymatrix{
\mathrm{lim}_L (U_{\alpha} \times_U U_{\alpha'}) \ar[r] \ar[d] & \mathrm{lim}_L U_{\alpha} \ar[d] \\
\mathrm{lim}_L U_{\alpha'} \ar[r] & \mathrm{lim}_L U.
}
$$
Note that $\alpha \neq \alpha'$ holds if and only if the pullback $U_{\alpha} \times_U U_{\alpha'}$ takes the value $\mathbf{0}$ for some $\ell \in L$. In this case, the limit of $U_{\alpha} \times_U U_{\alpha'}$ is $\mathbf{0}$, because there are no morphisms to $\mathbf{0} \in \Pre(\C)$ except for the identity. 

We now prove (d). Using (a), it follows that the induced morphism 
$$\mathrm{lim}_L U_{\alpha} \to \mathrm{lim}_L U$$
is a monomorphism for each $\alpha \in \P(U)$, and (b) states that these subobjects are pairwise disjoint. As a consequence, the canonical morphism
\begin{equation} \label{comparison_map} \tag{*}
\bigsqcup_{\alpha \in \P(U)} \mathrm{lim}_L U_{\alpha} \longrightarrow \mathrm{lim}_L U
\end{equation}
is again a monomorphism and it remains to prove that it is also an effective epimorphism in $\Pre(\C)$. Let $\R$ denote the collection of finite simplicial sets $L$ for which this holds for any diagram $U$ of the required form \eqref{coprod_rep}. In other words, the morphism \eqref{comparison_map} is an equivalence for $L \in \R$ and any $U$ satisfying \eqref{coprod_rep}. Then $\R$ has the following properties:
\begin{itemize}
\item[(i)]  $\Delta^n \in \R$ for any $n \geq 0$. Note that since $\Delta^n$ has an initial object, this claim is easily reduced to the obvious case of $\Delta^0$.
\item[(ii)] $\R$ contains the empty diagram $L=\varnothing$. In this case, the unique $L$-diagram is vacuously of the form \eqref{coprod_rep} and this diagram has a unique special point $\alpha$, which is an isomorphism, hence the comparison morphism \eqref{comparison_map} is also an isomorphism.  %$\R$ is closed under finite coproducts. If $L_1, L_2 \in \R$ and $U \colon L_1 \sqcup L_2 \to \Pre(\C)$ is a diagram of the required form \eqref{coprod_rep}, then the set of special points $\P(U)$ is the product of the special points of $U_{|L_i}$, $i=1,2$, and the morphism \eqref{comparison_map} is simply the product of the corresponding morphisms associated to $U_{|L_1}$ and $U_{|L_2}$. 
\item[(iii)] $\R$ is closed under pushouts
$$
\xymatrix{
L_0 \ar[r] \ar@{>->}[d] & L_2 \ar@{>->}[d] \\
L_1 \ar[r] & L
}
$$
where $L_0 \subseteq L_1$ is an inclusion.
In this case, for any $U \colon L \to \Pre(\C)$ of the required form, the special points $\P(U)$ of $U$ fit in a pullback square (of sets)
$$
\xymatrix{
\P(U) \ar[r] \ar[d] & \P(U_{|L_1}) \ar[d] \\
\P(U_{|L_2}) \ar[r] & \P(U_{|L_0})
}
$$
and for each $\alpha = (\alpha_1, \alpha_2; \alpha_0)$ in $\P(U)$, we have a canonical pullback in $\Pre(\C)$
$$
\xymatrix{
\mathrm{lim}_L U_{\alpha} \ar[r] \ar[d] & \mathrm{lim}_{L_1} U_{\alpha_1} \ar[d] \\
\mathrm{lim}_{L_2} U_{\alpha_2} \ar[r] & \mathrm{lim}_{L_0} U_{\alpha_0}. 
}
$$ 
Therefore, the morphism \eqref{comparison_map} for $U$ is identified with the pullback of the corresponding morphisms for $U_{|L_i}$. (It might be helpful to consider first the case where $L_0 = \varnothing$ and $L=L_1 \sqcup L_2$. In this case, the set of special points $\P(U)$ is the product of the special points of $U_{|L_i}$, $i=1,2$, and the morphism \eqref{comparison_map} is simply the product of the corresponding morphisms associated to $U_{|L_1}$ and $U_{|L_2}$.)
\end{itemize}
It follows from these properties that $\R$ contains all finite simplicial sets as required (from (ii) and (iii) it follows that $\R$ is closed under finite coproducts; in particular, $\R$ contains $\partial \Delta^1=\Delta^0 \sqcup \Delta^0 $; then we can obtain every finite simplicial set from these closure properties by induction on the dimension using (i) and (iii)). This completes the proof of (d).
\end{proof}

\subsection{Proof of Theorem \ref{Thm2}}
The ``only if'' direction was shown in Proposition \ref{flat_vs_cov-flat}. For the converse, we will apply the criterion of Corollary \ref{criterion_flatness2} using the semi-simplicial hypercoverings in $\Pre(\C)$ constructed in Subsection \ref{criteria_flatness}.

\smallskip

\noindent Let $f \colon \C \to \X$ be a covering-flat functor. Let $K$ be a finite simplicial set and let $$K \xrightarrow{g} \C \xrightarrow{j} \Pre(\C)$$
be a diagram of representable presheaves. We recall the notation $Z(g)$ for the limit of this diagram in $\Pre(\C)$ and $X(g)$ for the limit in $\X$ of the associated diagram 
$$K \xrightarrow{g} \C \xrightarrow{f} \X.$$ 
Let $U^{Z(g)}_{\bullet}$ denote the semi-simplicial hypercovering in $\Pre(\C)_{/Z(g)}$ constructed in Subsection \ref{criteria_flatness} and let 
$$V_{\bullet} := L(f)_* U^{Z(g)}_{\bullet} \colon N(\Delta^{op}_<) \to \X_{/X(g)}$$
be the associated semi-simplicial object in $\X_{/X(g)}$. By Corollary \ref{criterion_flatness2}, it suffices to show that $V_{\bullet}$ is a semi-simplicial hypercovering, that is, it suffices to show that the canonical morphism in $\X_{/X(g)}$ $$V_n \to M_n(V_{\bullet})$$ is an effective epimorphism for any $n \geq 0$. Note that $M_{0}V_{\bullet}$ is a terminal object in $\X_{/X(g)}$.

We will use the decomposition shown in Lemma \ref{decomposition_points} in order to express the $n$-th matching objects of $U^{Z(g)}_{\bullet}$ and $V_{\bullet}$ in a convenient form. 
Let $L_n$ denote the (nerve of the) opposite subcategory of $(\Delta_<)_{/[n]}$ spanned by $([k] \hookrightarrow [n])$ with $k \neq n$. By definition, the $n$-th matching object $M_n (U^{Z(g)}_{\bullet})$ is the limit of the finite diagram
$$U[n] \colon L_n \to \Pre(\C)_{/Z(g)}$$
which sends $([k] \hookrightarrow [n] )$ to $U^{Z(g)}_k$. Since $Z(g)$ is the limit of $jg \colon K \to \Pre(\C)$, we have an equivalence
$$\Pre(\C)_{/Z(g)} \simeq \Pre(\C)_{/jg},$$
so the limit of $U[n]$ is identified with the limit of the adjoint diagram in $\Pre(\C)$, denoted by
$$U[n, g] \colon L_n \ast K \to \Pre(\C).$$
Note that the values of $U[n,g]$ are coproducts of representable presheaves. Therefore, by Lemma \ref{decomposition_points}, we have a canonical equivalence:
$$\bigsqcup_{\alpha \in \P(U[n,g])} \lim_{L_n \ast K} U[n,g]_{\alpha} \xrightarrow{\sim} \lim_{L_n \ast K} U[n,g] \simeq M_n (U^{Z(g)}_{\bullet})$$
where each diagram 
$$U[n,g]_{\alpha} \colon L_n \ast K \xrightarrow{U'[n,g]_{\alpha}} \C \xrightarrow{j} \Pre(\C)$$
is a finite diagram of representable presheaves. Similarly, the $n$-th matching object $M_n (V_{\bullet})$ is the limit of the finite diagram in $\X_{/X(g)}$,
$$V[n] \colon L_n \xrightarrow{U[n]} \Pre(\C)_{/Z(g)} \xrightarrow{L(f)_{/c_g}} \X_{/X(g)}, \ \ ([k] \hookrightarrow [n]) \mapsto V_k,$$
which can be identified with the limit of the adjoint diagram in $\X$, denoted by
$$V[n, g] \colon L_n \ast K \xrightarrow{U[n,g]} \Pre(\C) \xrightarrow{L(f)} \X.$$

\begin{claim} \label{decomposition_matching}
The canonical morphism 
$$\bigsqcup_{\alpha \in \P(U[n,g])} \lim_{L_n \ast K} (L(f) \circ U[n,g]_{\alpha}) \xrightarrow{\sim} \lim_{L_n \ast K} V[n,g] \simeq M_n (V_{\bullet})$$
is an equivalence in $\X$.
\end{claim} 

\smallskip

\noindent \emph{Proof of Claim \ref{decomposition_matching}}. 
We may choose a regular cardinal $\kappa$ so that the full subcategory $\X^{\kappa} \subset \X$ spanned by the $\kappa$-presentable (or $\kappa$-compact)  objects in $\X$ has the following properties:
\begin{itemize}
\item[(i)] the inclusion $\X^{\kappa} \hookrightarrow \X$ extends to a left exact localization functor $$L \colon \Pre(\X^{\kappa}) \to \X;$$
\item[(ii)] the functor $f \colon \C \to \X$ factors through $\X^{\kappa} \subset \X$. 
\end{itemize}
Then the functor $L(f)$ factors (up to equivalence) as a composition:
$$\Pre(\C) \xrightarrow{F} \Pre(\X^{\kappa}) \xrightarrow{L} \X$$
where $F$ is induced by $f \colon \C \to \X^{\kappa} \subset \X$.  Since $L$ preserves coproducts and finite limits, it suffices to verify the equivalence of the claim in $\Pre(\X^{\kappa})$, that is, it suffices to show that the canonical morphism
\begin{equation} \label{decomposition_general} \tag{**}
\bigsqcup_{\alpha \in \P(U[n,g])} \lim_{L_n \ast K} (F \circ U[n,g]_{\alpha}) \xrightarrow{\sim} \lim_{L_n \ast K} (F \circ U[n, g]) \simeq M_n (F \circ U_{\bullet})
\end{equation}
is an equivalence in $\Pre(\X^{\kappa})$. (Here $F \circ U_{\bullet}$ can equivalently be regarded as a semi-simplicial object in $\Pre(\X^{\kappa})_{/Fjg}$.)
Note that the values of $$W[n,g] := F \circ U[n,g]$$ are coproducts of representable presheaves, because the values of $U[n,g]$ in $\Pre(\C)$ are coproducts of representable presheaves and $F$ preserves small coproducts and representable presheaves. Then the equivalence \eqref{decomposition_general} will follow directly from Lemma \ref{decomposition_points} (applied to $W[n,g]$) as soon as we establish a bijective correspondence between the special points of $U[n,g]$ and of $W[n,g]$ and identify the corresponding subobjects via $F$. 

A subobject $U[n,g]_{\alpha}$ of $U[n,g]$ is given pointwise by a choice of a representable summand, so it determines a corresponding subobject $W[n,g]_{\beta} := F \circ U[n,g]_{\alpha}$ of $W[n,g]$ with the same property, therefore, a special point $\beta$ of $W[n,g]$. This correspondence $\alpha \mapsto \beta$ defines a function $\P(U[n,g]) \to \P(W[n,g])$ and it is immediate that the function is injective. Conversely, given a special point $(\beta \colon c(\1) \to W[n,g]^{\delta})$ of $W[n,g]$, we apply the right adjoint of $F$, denoted by $f^* \colon \Pre(\X^{\kappa}) \to \Pre(\C)$, and obtain a morphism in $\Pre(\C)$
$$f^*\beta \colon c(\1) \simeq f^*c(\1) \to f^*W[n,g]^{\delta} \simeq U[n,g]^{\delta}$$
where the last equivalence can be seen by direct inspection and the definition of the functors involved. Moreover, it is easy to verify that $\alpha : = f^*\beta$ is a special point of $U[n,g]$ and therefore there is an associated subobject $U[n,g]_{\alpha}$ whose values are representable presheaves. Using the commutative square of diagrams in $\Pre(\C)$, 
$$
\xymatrix{
U[n,g]_{\alpha} \ar[r] \ar[d] & f^*W[n,g] \ar[d] \\
f^*c(\1) \ar[r]^(.45){f^*\beta} & f^* W[n,g]^{\delta}
}
$$
where the top arrow is defined using the unit of the adjunction, we obtain (by adjunction) a morphism $F \circ U[n,g]_{\alpha} \to W[n,g]_{\beta}$ between diagrams in $\Pre(\X^{\kappa})$ over the diagram $W[n,g]$. Both $W[n,g]_{\beta}$ and $F \circ U[n,g]_{\alpha}$ are subobjects of $W[n,g]$, each of which corresponds pointwise to a choice of a representable summand, therefore they must be equivalent. This completes the proof of the claim.  
\qed 

\medskip

\noindent We may now continue with the proof of Theorem \ref{Thm2}. For any $\alpha \in \P(U[n,g])$, we denote the corresponding summand of the decomposition of $M_n(V_{\bullet})$ by 
$$y_{\alpha} = \mathrm{lim}(V[n,g]_{\alpha} \colon L_n \ast K \xrightarrow{U'[n,g]_{\alpha}} \C \xrightarrow{f} \X).$$
Moreover, using the decomposition of Claim \ref{decomposition_matching} and universality of coproducts, we may identify $V_n$ with a coproduct of objects in $\X$,
$$V_n \simeq \bigsqcup_{\alpha} V_{n, \alpha}$$
so that the morphism $V_n \to M_n(V_{\bullet})$ is identified with the coproduct of morphisms $V_{n,\alpha} \to y_{\alpha}$. Then it suffices to show that each of these morphisms (for every $\alpha$) is an effective epimorphism since coproducts of effective epimorphisms are effective epimorphisms (see \cite[6.2.3.11]{HTT}).

Since $f \colon \C \to \X$ is covering-flat, there is a covering familiy $\{y'_i \to y_a\}_{i \in I_{\alpha}}$ with the property that for each $(y'_i \to y_{\alpha})$, there is a cone $(L_n \ast K)^{\triangleleft} \to \C$ on $U'[n,g]_{\alpha}$ with cone object $c_i \in \C$ and a factorization (of cones on $V[n,g]_{\alpha}$):
$$y'_i \to f(c_i) \to \lim_{L_n \ast K} V[n,g]_{\alpha} =y_{\alpha}.$$
Note that this composite admits a further factorization:
$$y'_i \to f(c_i) \to L(f)\big(\lim_{L_n \ast K} U[n,g]_{\alpha}\big) \xrightarrow{c_{U[n,g]_{\alpha}}} \lim_{L_n \ast K} V[n,g]_{\alpha} =y_{\alpha}.$$
The latter factorization shows that each composite morphism $(f(c_i) \to y_{\alpha})$ factors (up to homotopy) through the morphism $V_{n, \alpha} \to y_{\alpha}$; this observation uses the definition of $U^{Z(g)}_{\bullet}$. 

Since the canonical morphism $\bigsqcup_{i \in I_{\alpha}} y'_i \to y_{\alpha}$ is an effective epimorphism, the associated morphism 
$$\bigsqcup_{i \in I_{\alpha}} f(c_i) \to y_{\alpha}$$
is also an effective epimorphism \cite[6.2.3.12]{HTT}. Then it follows similarly that the morphism $V_{n, \alpha} \to y_{\alpha}$ is an effective epimorphism as required. This completes the proof of Theorem \ref{Thm2}. \qed

\begin{example}
By Theorem \ref{Thm2},  the colimit functor $\mathrm{colim} \colon \Pre(\C) \to \Sp$ is left exact if and only if the constant functor $\C \to \Sp$ at the terminal object $\ast \in \Sp$ is covering-flat. It is easy to see that the condition of being covering-flat in this case corresponds exactly to the condition that $\C^{\op}$ is a filtered $\infty$-category. Indeed, by Lemma~\ref{lem_cov-flat}, we may reduce to the case where $h^{\triangleleft}$ is a limit cone, which in this case must be constant at the terminal object. The condition of being covering-flat thus boils down to the fact that for any finite $K$ and any $g \colon K \rightarrow \C$, there exists a cone $K^{\triangleleft} \to \C$ on $g$.
\end{example}

\begin{example}\label{example_category_of_elements_filtered}
 Let $\C$ be a small $\infty$-category. Using Theorem~\ref{Thm2}, we can show that a functor $f \colon \C \rightarrow \Sp$ has a left exact left Kan extension $\Pre(\C) \rightarrow \Sp$ along the Yoneda embedding if and only if its category of elements, that is, the slice $\infty$-category $f_{\ast \slash}$, is cofiltered.
 
  First assume that $f$ is covering-flat and let $K$ be a finite simplicial set. A $K$-diagram in $f_{\ast \slash}$ corresponds to a diagram $g \colon K \rightarrow \C$ and a cone $h^{\triangleleft}$ on $h=fg$ with cone object $\ast \in \Sp$. A collection of morphisms $\{y^{\prime} \rightarrow \ast\}$ defines a covering family if and only if one of the objects $y^{\prime}$ is non-empty. Since $f$ is covering-flat, the sieve $S_{g,h^{\triangleleft}}$ contains a covering family, so it must also contain $\id_{\ast} \colon \ast \rightarrow \ast$. But this implies that the cone $h^{\triangleleft}$ factors through the image of some cone $g^{\triangleleft} \colon K^{\triangleleft} \rightarrow \C$. The resulting morphism $\ast \rightarrow f(c)$, where $c$ denotes the cone object of $g^{\triangleleft}$, yields an object in $f_{\ast \slash}$ which is the cone object of a cone in $f_{\ast \slash}$ on the original $K$-diagram. Thus $f_{\ast \slash}$ is cofiltered, as claimed.
 
 Conversely, assume that $f_{\ast \slash}$ is cofiltered. Let $K$ be a finite simplicial set, $g \colon K \rightarrow \C$ a diagram, and let $h^{\triangleleft} \colon K^{\triangleleft} \rightarrow \Sp$ be a cone on $h=fg$ with cone object $y$. There is a covering family $\{\ast \xrightarrow{u_i} y\}_{i \in I}$ 
 in $\Sp$, so it suffices to check that any cone $h^{\triangleleft}$ on $h=fg$ with cone object $\ast$ factors through a cone in the image of $f$. Such a cone $h^{\triangleleft}$ is precisely a $K$-diagram in $f_{\ast \slash}$, while such a factorization corresponds to a cone on this diagram in $f_{\ast \slash}$. Since $f_{\ast \slash}$ is cofiltered, the desired factorization of $h^{\triangleleft}$ does indeed exist, so $f$ is covering-flat.
 
This characterization of flat functors $f \colon \C \to \Sp$ can be useful more generally for functors $\C \to \X$ where $\X$ is an $\infty$-topos with enough points $\{x_a^* \colon \X \to \Sp\}_{a \in A}$. In this case, a functor $f \colon \C \to \X$ is flat if and only if the composite functor $\C \xrightarrow{f} \X \xrightarrow{x_a^*} \Sp$ is flat for every $a \in A$.
\end{example}

\section{Generalizations to $n$-topoi} \label{sec:n-topoi}

\subsection{From $\infty$-topoi to $n$-topoi} The question about the left exactness of $L(f)$ can be considered more generally in the context of $n$-topoi. Let $\C$ be a small $n$-category, where $1 \leq n \leq \infty$. The associated $n$-topos of $(n\text{-}1)$-truncated objects in $\Pre(\C)$ will be denoted by $\Pre(\C)_n$ (see \cite[6.4]{HTT}). More explicitly, this is the $\infty$-category of presheaves on $\C$ with values in the $\infty$-category of $(n\text{-}1)$-truncated spaces. We denote the associated localization/truncation functor by 
$$\tau_n \colon \Pre(\C) \to \Pre(\C)_n;$$
this is the left adjoint of the inclusion functor $i_n \colon \Pre(\C)_n \hookrightarrow \Pre(\C)$ (see \cite[5.5.6]{HTT} for details about truncated objects). Since $\C$ is an $n$-category by assumption, the Yoneda embedding $j \colon \C \to \Pre(\C)$ factors through the inclusion of the full subcategory $\Pre(\C)_n \subset \Pre(\C)$ (see \cite[2.3.4]{HTT} for details about $n$-categories). We denote the associated Yoneda embedding by $j_n \colon \C \to \Pre(\C)_n$. 

\smallskip

Let $\Y$ be a cocomplete $\infty$-category and let $f \colon \C \to \Y$ be a functor. We write $$L(f)_n \colon \Pre(\C)_n \to \Y$$ for the restriction of $L(f)$ to $\Pre(\C)_n$ along the inclusion functor $i_n$. The functor $L(f)_n$ is the left Kan extension of $f$ along the Yoneda embedding $j_n \colon \C \to \Pre(\C)_n$ and can also be identified with the left Kan extension of $L(f)$ along the truncation functor $\tau_{n} \colon \Pre(\C) \to \Pre(\C)_n$. 
$$
\xymatrix{
& \C \ar[rrd]^f \ar[dl]_{j_n} \ar[d]^{j} && \\
\Pre(\C)_n \ar@{^{(}->}[r]_{i_n} \ar@{=}[rd] & \Pre(\C) \ar[d]_{\tau_n} \ar@{-->}[rr]^{L(f)}   \ar@{}[drr]^(0.45){\Downarrow} && \Y. \\
& \Pre(\C)_n \ar@{-->}[rru]_{L(f)_n} &&
}
$$
(We will keep the notation $L(f)$ for $L(f)_{\infty}$ and $\Pre(\C)$ for $\Pre(\C)_n$ when $n=\infty$.) 

\begin{lemma} \label{lem_adjunction}
Let $\C$ be a small $n$-category and let $\Y$ be a cocomplete $n$-category. For a functor $f \colon \C \to \Y$, the associated functor $L(f)_n \colon \Pre(\C)_n \to \Y$ is a left adjoint. In particular, $L(f)_n$ preserves small colimits.
\end{lemma}
\begin{proof}
Let $G \colon \Y \to \Pre(\C)$ denote the right adjoint of $L(f)$. Using this adjunction and the Yoneda lemma \cite[5.1.3, 5.2.6.5]{HTT}, \cite[5.8.9]{Ci}, for any $y \in \Y$ and $c \in \C$, we have canonical equivalences:
$$G(y)(c) \simeq \map_{\Pre(\C)}(j(c), G(y)) \simeq \map_{\Y}(f(c), y).$$
Therefore the values of $G(y)$ are $(n\text{-}1)$-truncated for every $y \in \Y$. This implies that $G$ factors through $i_n \colon \Pre(\C)_n \subset \Pre(\C)$ -- this factorization is also a consequence of the fact that $G$ is left exact, see \cite[5.5.6.16]{HTT}. It follows that $G \colon \Y \to \Pre(\C)_n$ is a right adjoint of $L(f)_n$. 
\end{proof}

\begin{remark}
Lemma \ref{lem_adjunction} fails in general when $\Y$ is not an $n$-category. For example, in the case where $f$ is the Yoneda embedding $j \colon \C \to \Pre(\C)$, then $L(j) \colon \Pre(\C) \to \Pre(\C)$ is the identity functor and $L(j)_n \colon \Pre(\C)_n \to \Pre(\C)$ is  simply the inclusion functor. This inclusion functor does not admit a right adjoint for any $n < \infty$ (unless $\C$ is the empty category). 
\end{remark}

 From the above lemma it follows that $\Pre(\C)_n$ is the free cocomplete $n$-category on $\C$. For a cocomplete $n$-category $\Y$, we denote by $\mathrm{Fun}^{L}(\Pre(\C)_n,\Y)$ the full subcategory of left adjoint functors. 
 
 \begin{proposition}\label{prop:free_cocomplete_n_category}
Let $\C$ be a small $n$-category and let $\Y$ be a cocomplete $n$-category. Then the functor
 \[
 \mathrm{Fun}^{L}(\Pre(\C)_n,\Y) \rightarrow \mathrm{Fun}(\C,\Y)
 \]
 given by restriction along the Yoneda embedding $j_n \colon \C \rightarrow \Pre(\C)_n$ is an equivalence. Its inverse sends $f$ to $L(f)_n$.
 \end{proposition}
\begin{proof}
 We know from \cite[5.1.5.6]{HTT} that restriction along $j \colon \C \rightarrow \Pre(\C)$ induces an equivalence $\mathrm{Fun}^{L}(\Pre(\C),\Y) \rightarrow \mathrm{Fun}(\C,\Y)$ whose inverse sends $f$ to $L(f)$. Moreover, passing to the right adjoint yields an equivalence
 \[
 \mathrm{Fun}^{L}(\Pre(\C),\Y) \rightarrow \mathrm{Fun}^{R}(\Y,\Pre(\C))^{\op},
 \]
 where we write $\Fun^{R}$ for the full subcategory of right adjoint functors. Since right adjoint functors preserve $(n\text{-}1)$-truncated objects, each such right adjoint factors through the (right adjoint) inclusion $i_n \colon \Pre(\C)_n \rightarrow \Pre(\C)$. From this it follows that composition with $i_n$ gives equivalences of $\infty$-categories:
 \[
 \mathrm{Fun}^{R}(\Y,\Pre(\C)_n) \rightarrow \mathrm{Fun}^R(\Y,\Pre(\C)) \quad \text{and} \quad \mathrm{Fun}^L(\Pre(\C),\Y) \rightarrow \mathrm{Fun}^{L}(\Pre(\C)_n,\Y)
 \]
Thus we find that the composite
 \[
 \xymatrix{\mathrm{Fun}(\C,\Y) \ar[r]^-{L(-)} & \mathrm{Fun}^{L}(\Pre(\C),\Y) \ar[r]^-{i_n^{\ast}} & \mathrm{Fun}^L(\Pre(\C)_n,\Y)}
 \]
 is an equivalence. Since $L(f) i_n j_n \simeq L(f) j \simeq f$, we conclude that its inverse is given by restriction along the Yoneda embedding $j_n \colon \C \rightarrow \Pre(\C)_n$.
\end{proof}

\begin{remark}
The $\infty$-category $\Fun^L(\Pre(\C)_n, \Y)$ agrees with the $\infty$-category of functors which preserve small colimits (cf. \cite[4.1.4]{NRS}). This follows from the fact that a colimit-preserving functor $F\colon \Pre(\C)_n \to \Y$ can be identified with the left adjoint functor $L(Fj_n)_n$. To see this, note that there is a canonical natural transformation $L(Fj_n)_n \to F$ between colimit-preserving functors, both of which are extensions of $Fj_n \colon \C \to \Y$, and each object in $\Pre(\C)_n$ is canonically a small colimit of representable presheaves.  
\end{remark}

The following proposition shows that the question about the left exactness of $L(f)_n$ in the context of $n$-topoi ($1 \leq n \leq \infty$) can be reduced to the case $n=\infty$. 

\begin{proposition} \label{prop_comparison}
Let $\C$ be a small $n$-category, where $1 \leq n < \infty$, and let $\X$ be an $\infty$-topos. Let $\X_{n}$ denote the $n$-topos of $\text{(n--1)}$-truncated objects in $\X$ and let $\tau^{\X}_n \colon \X \to \X_n$ denote the localization functor which is left adjoint to the inclusion functor $i^{\X}_{n} \colon \X_n \subset \X$. 

For a functor $f \colon \C \to \X_n$, the induced functor $L(f)_n \colon \Pre(\C)_n \to \X_n$ of $n$-topoi is left exact if the associated functor $L(i^{\X}_{n} f) \colon \Pre(\C) \to \X$ of $\infty$-topoi is left exact. 
\end{proposition} 
\begin{proof} If the left adjoint $L(i^{\X}_{n} f)$ is left exact, then both $L(i^{\X}_n f)$ as well as its right adjoint preserve $(n\text{-}1)$-truncated objects (see \cite[5.5.6.16]{HTT}), and therefore the (truncated) functor $L(i^{\X}_{n} f)_{| \Pre(\C)_n} \colon \Pre(\C)_n \to \X_n$ is again a left exact left adjoint functor. But this agrees with $L(f)_n$, since both are left adjoints extending $\C \xrightarrow{f} \X_n$ (Lemma \ref{lem_adjunction}), so they are equivalent by Proposition~\ref{prop:free_cocomplete_n_category}.
\end{proof}

We will prove a partial converse of Proposition \ref{prop_comparison} at the end of the section. 

\subsection{Diaconescu's theorem for $n$-topoi} Proposition \ref{prop_comparison} implies that any criterion for the left exactness of the functor $L(i^{\X}_{n} f) \colon \Pre(\C) \to \X$ in the context of $\infty$-topoi automatically yields a criterion for the left exactness of $L(f)_n \colon \Pre(\C)_n \to \X_n$ in the context of $n$-topoi for any $n \geq 1$. (We note that the left exactness of $L(f)_n$ does not imply that $L(f) \colon \Pre(\C) \to \X_n$ is left exact. This fails, for example, when $f$ is simply the Yoneda embedding $j_n \colon \C \to \Pre(\C)_n$, in which case $L(f)_n$ is the identity functor, but $L(f) \simeq \tau_n$ is not left exact.)

\smallskip

Let $\C$ be a small $n$-category and let $\Y$ be an $n$-topos, where $1 \leq n < \infty$. The following theorem gives necessary and sufficient conditions for a functor $f \colon \C \to \Y$ to extend to a left adjoint $L(f)_n \colon \Pre(\C)_n \to \Y$ which is left exact. This result  generalizes the characterizations of Theorem \ref{Thm1} and Theorem \ref{Thm2} to $n$-topoi for any $n \geq 1$. 

\begin{theorem} \label{Thm3}
Let $\C$ be a small $n$-category and let $\Y$ be an $n$-topos, $1 \leq n < \infty$. For a functor $f \colon \C \to \Y$, the following are equivalent:
\begin{enumerate}
\item $L(f)_n \colon \Pre(\C)_n \to \Y$ is left exact; 
\item $L(f)_n\colon \Pre(\C)_n \to \Y$ preserves finite limits of representable presheaves, that is, $L(f)_n$ preserves limits of diagrams of the form 
$$(K \to \C \to \Pre(\C)_n)$$ 
where $K$ is a finite simplicial set;
\item $f\colon \C \to \Y$ is covering-flat (see Definition \ref{covering_condition}).
\end{enumerate}
\end{theorem}
\begin{proof}
(1) $\Rightarrow$ (2) is obvious. (2) $\Rightarrow$ (3) is analogous to Proposition \ref{flat_vs_cov-flat}, but the proof requires a small modification in this context so that it does not depend directly on Theorem \ref{Thm1}  (cf. Remark \ref{directproof_n-topoi}).  Let $(K, g \colon K \to \C, h^{\triangleleft} \colon K^{\triangleleft} \to \Y, y)$ be as in Definition \ref{covering_condition}(b) where $h^{\triangleleft}$ is a limit cone (using Lemma \ref{lem_cov-flat}). Assuming that $f \colon \C \to \Y$ satisfies (2), we may identify $h^{\triangleleft}$ with the cone $L(f)_n \circ (j_n g)^{\triangleleft}$ where $(j_n g)^{\triangleleft} \colon K^{\triangleleft} \to \Pre(\C)_n$ is a limit cone on $j_n g \colon K \to \C \to \Pre(\C)_n$ with cone object denoted by $Z \in \Pre(\C)_n$. Let $\{q_i \colon j_n(c_i) \rightarrow Z\}_{i \in I}$ be a covering family of $Z$ by representable presheaves.

\begin{claim} \label{Claim3} The image under $L(f)_n$ of the covering family $\{q_i \colon j_n(c_i) \rightarrow Z\}_{i \in I}$ is a covering family of $L(f)_n(Z)=y$ in $\Y$.
\end{claim}
\noindent \emph{Proof of Claim \ref{Claim3}}. It will be convenient to assume that $\Y$ is the $n$-topos $\X_n$ associated to an $\infty$-topos $\X$ \cite[6.4.1.5]{HTT} and write $i^{\X}_n \colon \X_n \to \X$ for the inclusion functor and $\tau_n^{\X} \colon \X \to \X_n$ for its left adjoint/truncation functor.  

Let $U_{\bullet}$ be a \v{C}ech nerve of the effective epimorphism 
$$q \colon \bigsqcup_{i \in I} j_n(c_i) \longrightarrow Z,$$
defined by the maps $q_i \colon j_n(c_i) \to Z$.  Then consider the augmented simplicial object $V_{\bullet} := L(f)_n \circ U_{\bullet}$ in $\X_n$. Similarly to the proof of Theorem \ref{Thm1}, we note that $L(f)_n$ is left exact at the representable presheaves (in the sense of \cite[6.1.4.1]{HTT}) and therefore it is also left exact at small coproducts of representable presheaves (using \cite[6.1.5.1]{HTT}). Then the underlying simplicial object of $V_{\bullet}$ is again a groupoid object (see \cite[6.1.4.3]{HTT}), and consequently, so is also the (underlying) simplicial object $i_n^{\X} V_{\bullet}$ in $\X$. 
It follows that the canonical morphism in $\X$ 
$$i_n^{\X}(V_0) \longrightarrow \mathrm{colim}_{N(\Delta^{op})} i_n^{\X} V_{\bullet}$$
is an effective epimorphism and $i_n^{\X} V_{\bullet}$ is its \v{C}ech nerve. Applying $\tau_n^{\X}$, we conclude that the morphism
\begin{equation} \label{augmentation_map} \tag{$\star$}
V_0 \to \mathrm{colim}_{N(\Delta^{op})} V_{\bullet}
\end{equation}
is an effective epimorphism in $\X_n$ since effective epimorphisms can be detected at the level of the ordinary $(1\text{-})$-topos \cite[7.2.1.14]{HTT}. On the other hand, since $q$ is an effective epimorphism in $\Pre(\C)_n$, it follows that the augmented simplicial object $U_{\bullet}$ is a colimit diagram, therefore so is $V_{\bullet}$ in $\X_n$, because $L(f)_n$ preserves colimits. This means that $\mathrm{colim}_{N(\Delta^{op})} V_{\bullet} \simeq V_{-1} = y$ and the effective epimorphism \eqref{augmentation_map} can be identified with the morphsm $L(f)_n(q) \colon \bigsqcup_{i \in I} f(c_i)\to y$.  This completes the proof of the claim. \qed

\smallskip

Note that each morphism $(j_n(c) \to Z)$ corresponds (using the Yoneda embedding) to a cone $g^{\triangleleft} \colon K^{\triangleleft} \rightarrow \C$ on $g$. Therefore the covering family of Claim \ref{Claim3}
$$\{L(f)_n(q_i) \colon f(c_i) \to y\}_{i \in I}$$ 
is in $S_{g, h^{\triangleleft}}$ as required.

(3) $\Rightarrow$ (1):  We may assume that $\Y$ is the $n$-topos $\X_n$ associated to an $\infty$-topos $\X$ \cite[6.4.1.5]{HTT}. We write $i^{\X}_n \colon \X_n \to \X$ for the inclusion functor and $\tau_n^{\X} \colon \X \to \X_n$ for its left adjoint/truncation functor.  The proof is based on the following crucial observation:

\begin{claim} \label{Claim2} The functor $f\colon \C \to \X_n$ is covering-flat if and only if $i^{\X}_n f \colon \C \to \X$ is covering-flat. 
\end{claim}
\noindent \emph{Proof of Claim \ref{Claim2}}.  Let $g \colon K \to \C$ be a diagram, where $K$ is a finite simplicial set, and let $h^{\triangleleft} \colon K^{\triangleleft} \to \X_n$ be a limit of the diagram $h = f g \colon K \to \X_n$ with cone object $y \in \X_n$. Since the inclusion $i_n^{\X}\colon \X_n \subset \X$ preserves small limits, the cone $i_n^{\X} h^{\triangleleft} \colon K^{\triangleleft} \to \X$ is a limit also in $\X$. If $f \colon \C \to \X_n$ is covering-flat, then there is a set $S$ of morphisms $\{y'_i \to y\}_{i \in I}$ in $S_{g, h^{\triangleleft}}$ such that the canonical map in $\X_n$
\begin{equation} \label{eff_epi} \tag{*}
\bigsqcup_{i \in I} y'_i \longrightarrow y
\end{equation}
is an effective epimorphism.  The corresponding set $S'= i^{\X}_n(S)$ of morphisms $\{i_n^{\X}(y'_i) \to i_n^{\X}(y)\}_{i \in I}$ is clearly a subset of $S_{g, i^{\X}_n h^{\triangleleft}}$. Then it suffices to show that the canonical map in $\X$
$$\bigsqcup_{i \in I} i_n^{\X}(y'_i) \to i_n^{\X}(y)$$
is an effective epimorphism. Since being an effective epimorphism is detected at the level of the ordinary $(1\text{-})$topos \cite[7.2.1.14]{HTT}, this follows from the fact that \eqref{eff_epi} is an effective epimorphism. 

The converse is similar. If $i_n^{\X} f \colon \C \to \X$ is covering-flat, then there is a set $T$ of morphisms $\{y'_i \to i_n^{\X}(y)\}_{i \in J}$ in $S_{g, i_n^{\X}h^{\triangleleft}}$ such that the canonical map in $\X$
\begin{equation} \label{eff_epi2} \tag{**}
\bigsqcup_{i \in J} y'_i \longrightarrow i_n^{\X}(y)
\end{equation}
is an effective epimorphism.  Then the corresponding set $T'= \tau^{\X}_n(T)$ of morphisms $\{\tau_n^{\X}(y'_i) \to y)\}_{i \in J}$ (using $\tau_n^{\X} i_n^{\X}(y) \simeq y$) is a subset of $S_{g, h^{\triangleleft}}$ and the canonical map 
$$\bigsqcup_{i \in J} \tau_n^{\X}(y'_i) \to y$$
is again an effective epimorphism in $\X_n$ because effective epimorphisms are detected at the level of the ordinary $(1\text{-})$topos \cite[7.2.1.14]{HTT} (alternatively, it suffices to note that \eqref{eff_epi2} factors canonically through the last morphism and apply \cite[6.2.3.12]{HTT}). This completes the proof of the claim. \qed

\smallskip

\noindent To wrap up the proof, we now also assume without loss of generality that $\X$ is hypercomplete. This is possible since the $(n\text{-}1)$-truncated objects in an $\infty$-topos are hypercomplete. Assuming (3), it follows that the functor $L(i_n^{\X}f) \colon \Pre(\C) \to \X$ is left exact by applying Claim \ref{Claim2} and Theorem \ref{Thm2}. Therefore the functor $L(f)_n \colon \Pre(\C)_n \to \X_n$ is also left exact by Proposition \ref{prop_comparison}. This completes the proof of (3) $\Rightarrow$ (1).
\end{proof}

\begin{definition}[Flat functors] \label{def:flat_functor2}
Let $\C$ be a small $n$-category and let $\Y$ be an $n$-topos, where $1 \leq n < \infty$. A functor $f \colon \C \to \Y$ is called \emph{flat} if it satisfies the equivalent conditions of Theorem \ref{Thm3}. We denote by $\Fun^{\fl}(\C, \Y)$ the $n$-category of flat functors. 
\end{definition}

\begin{remark}
 In the case $n=1$, Theorem~\ref{Thm3} recovers the theorem that filtering functors and flat functors coincide (see \cite[VII.9.1]{MM}). Indeed, by unraveling the definitions of the categories and cones appearing in \cite[VII.8.4]{MM}, we immediately see that a functor from an ordinary (1-)category to an ordinary (1-)topos is filtering if and only if it is covering-flat in the sense of Definition~\ref{covering_condition}.
\end{remark}

\begin{example} \label{example_groupoid}
Let $G$ be (the nerve of) an ordinary group and let $$\colim \colon \Pre(G)_1 \simeq \Fun(G^{\op}, \Set) \to \Set$$ be the colimit functor. This is induced by the diagram $f \colon G \to \Set$ that is given by the one-point set with the trivial $G$-action, that is, $\colim \simeq L(f)_1$. Note that  $L(f)_1$ preserves the terminal object and pullbacks of representable presheaves (as well as monomorphisms), but $L(f)_1$ does not preserve products of representable presheaves -- indeed this happens only when $f$ is the corepresentable functor, i.e., the free $G$-set $G$ with the natural (left) $G$-action. This example shows that condition (2)$'$ of Theorem \ref{Thm1} is not sufficient in the context of $n$-topoi for $n < \infty$. 
\end{example}

\begin{remark} \label{directproof_n-topoi}
It is interesting to compare the proofs of the implication (2) $\Rightarrow$ (1) for $n=\infty$ (Theorem \ref{Thm1}) and for $n<\infty$ (Theorem \ref{Thm3}). 
Compared to the proof in Theorem \ref{Thm1}, the proof of Theorem \ref{Thm3}(2)$\Rightarrow$(1) is indirect and employs some of our previous results so far. Indeed the proof is based on Theorem \ref{Thm2} (via Claim \ref{Claim2}) and Proposition \ref{prop_comparison}. The proof in Theorem \ref{Thm1} does not apply here directly, because it would be needed to show in addition that the groupoid $L(f)(U_{\bullet})$ (as defined in the proof of Theorem \ref{Thm1}) is also $n$-\emph{efficient}. This is possible assuming (2), however, the proof does not seem to be straightforward. On the other hand, using this approach, it is also possible to identify specific shapes of diagrams that would suffice for the left exactness of $L(f)_n$ as in Theorem \ref{Thm1}(2)$'$ (cf. Remark \ref{example_groupoid}), but we will not pursue this further. Moreover, we think that the proof method of Theorem \ref{Thm3} by \emph{reduction to} $\infty$, where higher topos theory admits its most compact form, is a natural approach to statements about $n$-topoi for any finite $n \geq 1$. 
\end{remark}

We have the following extension of Corollary \ref{cor:flat_functors} to $n$-topoi for $n < \infty$. 

\begin{corollary} \label{cor:flat_functors2}
Let $\C$ be a small $n$-category and let $\Y$ be an $n$-topos, $1 \leq n < \infty$. Then restriction along the Yoneda embedding $j_n \colon \C \to \Pre(\C)_n$ defines an equivalence of $\infty$-categories  
$$\Fun^{L, \lex}(\Pre(\C)_n, \Y) \xrightarrow{j_n^*} \Fun^{\fl}(\C, \Y).$$
In other words, the $n$-topos $\Pre(\C)_n$ is the classifying $n$-topos for flat functors. 
\end{corollary}

Similarly, we obtain the following extension of Corollary \ref{cor:finite_lim} to $n$-topoi for $n < \infty$. 

\begin{corollary}
Let $\C$ be a small $n$-category with finite limits and let $\Y$ be an $n$-topos, $1 \leq n < \infty$. For a functor $f \colon \C \to \Y$, the following are equivalent:
\begin{enumerate}
\item $L(f)_n \colon \Pre(\C)_n \to \Y$ is left exact; 
\item $f$ preserves finite limits. 
\end{enumerate} 
\end{corollary}
\begin{proof}
(1) $\Rightarrow$ (2) follows from the fact that $j_n \colon \C \to \Pre(\C)_n$ preserves finite limits. (2) $\Rightarrow$ (1) follows from Theorem \ref{Thm3} using that $\C$ has finite limits and $j_n$ preserves them. 
\end{proof}

\begin{example}
Let $\G$ be a small $n$-groupoid with a single object and let $\Y=\Sp_n$ be the $n$-topos associated to the $\infty$-category of spaces, $1 \leq n \leq \infty$. A functor $f \colon \G \to \Sp_n$ is flat if and only if it is equivalent to the corepresentable functor (free $\G$-space $\G$). To see this, note that $f \colon \G \to \Sp_n$ is flat (= covering-flat) if and only if the composite functor $\G \xrightarrow{f} \Sp_n \xrightarrow{i_n} \Sp$ is flat (=covering-flat) and use Example \ref{infty_grpd}. Thus, the $n$-category $\Fun^{L, \lex}(\Pre(G)_n, \Sp_n) \simeq \Fun^{\fl}(\G, \Sp_n)$ is equivalent to $\G$. 
\end{example}

\begin{remark}
 Let $\C$ be a small $n$-category. As in Example~\ref{example_category_of_elements_filtered}, we can show that the left Kan extension of $f \colon \C \rightarrow \Sp_n$ along the Yoneda embedding is left exact if and only if the category $f_{\ast \slash}$ of elements is cofiltered. The argument given in Example~\ref{example_category_of_elements_filtered} only relied on the two facts that a morphism $\bigsqcup_{y^{\prime} \in I} y^{\prime} \rightarrow \ast$ is an effective epimorphism if and only if one of the $y^{\prime}$ is non-empty and that there is a covering family $\{\ast \xrightarrow{u_i} y\}_{i \in I}$ for any $y$ in $\Sp$. Both of these facts are also true in the category $\Sp_n$ of $(n\text{-}1)$-truncated spaces.
\end{remark}

\subsection{Some applications} \label{application} Using Corollary \ref{cor:flat_functors2}, we also obtain the following generalization of \cite[6.4.5.5]{HTT} to general small $n$-categories. 

\begin{corollary} \label{cor:flat_functors3}
Let $\C$ be a small $n$-category and let $\X$ be a hypercomplete $m$-topos, where $1 \leq n \leq m \leq \infty$. Then the restriction functor 
$$\Fun^{L, \lex}(\Pre(\C)_m, \X) \rightarrow \Fun^{L, \lex}(\Pre(\C)_n, \X_n)$$
is an equivalence of $\infty$-categories.
\end{corollary}
\begin{proof}  For $1 \leq n < m < \infty$, we may consider a hypercomplete $\infty$-topos associated to $\X$ and therefore we observe that it suffices to prove the statement for $m=\infty$. By Corollary \ref{cor:flat_functors} and Corollary \ref{cor:flat_functors2}, the vertical functors in the commutative diagram
$$
\xymatrix{
\Fun^{L, \lex}(\Pre(\C), \X) \ar[d]_{j^*} \ar[r] & \Fun^{L, \lex}(\Pre(\C)_n, \X_n) \ar[d]^{j_n^*} \\
\Fun^{\fl}(\C, \X) & \Fun^{\fl}(\C, \X_n) \ar[l]
}
$$
are equivalences. The bottom functor is defined by the composition with the inclusion $i_n^{\X} \colon \X_n \subset \X$; this is well-defined using Claim \ref{Claim2} and the fact that flat functors and covering-flat functors agree in these cases by Theorem \ref{Thm2} and Theorem \ref{Thm3}, respectively. Moreover, the bottom functor is also fully faithful since $i_n^{\X}$ is. The bottom functor is also essentially surjective on objects since the diagram commutes and $j^{*}$ is an equivalence. It follows that the top functor is fully faithful and essentially surjective, as required. 
(Note that we have used that $\C$ is an $n$-category in this diagram.)
\end{proof}

Corollary \ref{cor:flat_functors3} has the following immediate consequence for topological localizations of $\Pre(\C)$ where $\C$ is an $n$-category. 

\begin{corollary} \label{cor:flat_functors4} 
Let $\C$ be a small $n$-category with a Grothendieck topology $\tau$ and let $\Sh(\C, \tau)$ denote the associated $\infty$-topos of sheaves. For any hypercomplete $m$-topos $\X$, where $1 \leq n \leq m \leq \infty$, the restriction functor 
$$\Fun^{L, \lex}(\Sh(\C,\tau)_m, \X) \rightarrow \Fun^{L, \lex}(\Sh(\C, \tau)_n, \X_n)$$
is an equivalence of $\infty$-categories.
\end{corollary}
\begin{proof}
The proof follows \cite[6.4.5.6]{HTT}. Since $\C$ is an $n$-category, every representable presheaf is in $\Pre(\C)_n$ and therefore so are its subobjects. 
For any $m \geq n$, the $\infty$-category $\Fun^{L, \lex}(\Sh(\C, \tau)_m, \X)$ can be identified with a full subcategory of the $\infty$-category $\Fun^{L, \lex}(\Pre(\C)_m, \X)$. More specifically, this subcategory is spanned by the functors $F \colon \Pre(\C)_m \to \X$ with property that each monomorphism 
$U \to j(c)$ in $\Pre(\C)_m$ (in fact, in $\Pre(\C)_n$), which becomes an equivalence in $\Sh(\C, \tau)_m$, is sent by $F$ to an equivalence in $\X$. Then the claim follows from Corollary \ref{cor:flat_functors3}.
\end{proof}

 Passing to right adjoints shows that, under the above assumptions, the $\infty$-category of geometric morphisms $\X \rightarrow \Sh(\C,\tau)_m$ is equivalent to the $\infty$-category of geometric morphisms $\X_n \rightarrow \Sh(\C,\tau)_n$.

\begin{remark}
Assuming that $\C$ is a small $n$-category with finite limits, then the $\infty$-topos $\Pre(\C)$ is $n$-localic \cite[6.4.5; esp. 6.4.5.5]{HTT}. This fails in general if $\C$ does not have finite limits \cite[20.4.0.1]{SAG}, \cite[1.14]{Ha}. In other words, the claim of Corollary \ref{cor:flat_functors3} cannot be extended to all $\infty$-topoi $\X$ in general. Thus, the statements of Corollaries \ref{cor:flat_functors3} and \ref{cor:flat_functors4} suggest a property weaker than $n$-locality which holds only relative to hypercomplete $\infty$-topoi and is satisfied by the examples that arise from $n$-sites.
\end{remark}

\begin{remark}
Given an arbitrary small $n$-category $\C$, $1 \leq n < \infty$, and an arbitrary $\infty$-topos $\X$ with hypercompletion denoted by $\ell \colon \X \to \X^{\wedge}$, then composition with $\ell$ defines a fully faithful functor 
\begin{equation} \label{hypercompletion_flatness} \tag{$\wedge$}
\Fun^{L, \lex}(\Pre(\C), \X) \rightarrow \Fun^{L, \lex}(\Pre(\C), \X^{\wedge})
\end{equation}
but this is not an equivalence in general (since Corollary \ref{cor:flat_functors4} fails for general $\infty$-topoi $\X$). In more detail, since the representable presheaves are $n$-truncated in $\Pre(\C)$, the restriction along the Yoneda embedding $j \colon \C \to \Pre(\C)$ identifies the functor $\infty$-categories in \eqref{hypercompletion_flatness} with full subcategories of $\Fun(\C, \X_n)$. More specifically, we may identify 
$$\Fun^{L, \lex}(\Pre(\C), \X) \text{ (resp. } \Fun^{L, \lex}(\Pre(\C), \X^{\wedge})\text{)}$$ 
with the full subcategory of functors $f \colon \C \to \X_n$ for which $i_n^{\X} f \colon  \C \to \X$ (resp. $\ell i_n^{\X} f \colon \C \to \X^{\wedge}$) is flat. By Theorem \ref{Thm2} and Claim \ref{Claim2}, the following statements are equivalent:
\begin{itemize}
\item[(a)] the composite functor $\C \xrightarrow{f} \X_n \xrightarrow{i_n^{\X}} \X \xrightarrow{\ell} \X^{\wedge}$ is flat;
\item[(b)] $f \colon \C \to \X_n = (\X^{\wedge})_n$ is covering-flat; 
\item[(c)] the composite functor $\C \xrightarrow{f} \X_n \xrightarrow{i_n^{\X}} \X$ is covering-flat.
\end{itemize} 
On the other hand, the characteristic property for $\Fun^{L, \lex}(\Pre(\C), \X)$, i.e., 
\begin{itemize}
\item[(d)] the composite functor $\C \xrightarrow{f} \X_n \xrightarrow{i_n^{\X}} \X$ is flat,
\end{itemize}
is strictly stronger than (a)--(c) for general $\X$.
\end{remark}

Finally, we can now prove a partial converse of Proposition \ref{prop_comparison}.

\begin{corollary}  \label{prop_comparison2}
Let $\C$ be a small $n$-category, where $1 \leq n < \infty$, and let $\X$ be a hypercomplete $\infty$-topos. Let $\X_{n}$ denote the $n$-topos of $\text{(n--1)}$-truncated objects in $\X$ and let $\tau^{\X}_n \colon \X \to \X_n$ denote the localization functor which is left adjoint to the inclusion functor $i^{\X}_{n} \colon \X_n \subset \X$. 

For a functor $f \colon \C \to \X_n$, the induced functor $L(f)_n \colon \Pre(\C)_n \to \X_n$ of $n$-topoi is left exact if and only if the associated functor $L(i^{\X}_{n} f) \colon \Pre(\C) \to \X$ of $\infty$-topoi is left exact. 
\end{corollary}
\begin{proof} The ``if'' part of the claim was shown in Proposition \ref{prop_comparison}. For the converse, we recall from Corollary \ref{cor:flat_functors3} (for $m=\infty$) that there is an equivalence between $\infty$-categories of left exact left adjoint functors:
$$\Fun^{L, \lex}(\Pre(\C), \X) \simeq \Fun^{L, \lex}(\Pre(\C)_n, \X_n)$$
$$ (F \colon \Pre(\C) \to \X) \mapsto (F_{| \Pre(\C)_n} \colon \Pre(\C)_n \to \X_n).$$
So, if $L(f)_n$ is left exact, then there is an essentially unique left exact left adjoint functor $F \colon \Pre(\C) \to \X$ such that 
$F_{| \Pre(\C)_n} \simeq L(f)_n$. But then $F \simeq L(i^{\X}_{n} f)$, since both are colimit-preserving extensions of $\C \xrightarrow{f} \X_n \xrightarrow{i^{\X}_n} \X$.
\end{proof}

\end{document}